\documentclass[11pt]{amsart}
\usepackage{amssymb,amsthm}

\usepackage[]{graphicx}

\everymath{\displaystyle}

\newtheorem{theo}{Theorem}[section]
\newtheorem{prop}[theo]{Proposition}

\newtheorem{lemma}[theo]{Lemma}
\newtheorem{teo}[theo]{Theorem}
\newtheorem{lem}[theo]{Lemma}
\newtheorem{cor}[theo]{Corollary}
\theoremstyle{definition}
\newtheorem{defi}[theo]{Definition}
\newtheorem{theodefi}[theo]{Theorem-Definition}
\newtheorem{rema}[theo]{Remark}

\newtheorem{remark}[theo]{Remark}

\newcommand{\fin}{\hfill$\square$}
\newcommand{\eu}{\op{eu}}

\newcommand{\ADS}{\mathbb{ADS}}

\newcommand{\wt}{\widetilde}

\newcommand{\Cl}{\mbox{Cl}}

\newcommand{\sC}{\mathfrak C}

\newcommand{\cT}{\mathcal T}

\newcommand{\cW}{\mathcal W}
\newcommand{\cZ}{\mathcal Z}

\def\d{\mathrm{d}}

\newcommand{\DD}{\mathbb{D}}

\newcommand{\HH}{\mathbb{H}}

\newcommand{\NN}{\mathbb{N}}

\newcommand{\PP}{\mathbb{P}}

\newcommand{\RR}{\mathbb{R}}
\renewcommand{\SS}{\mathbb{S}}

\newcommand{\ZZ}{\mathbb{Z}}
\newcommand{\op}{\operatorname}

\DeclareMathOperator{\SO}{SO}
\DeclareMathOperator{\uSO}{\widetilde{SO}}

\DeclareMathOperator{\AdS}{AdS}
\DeclareMathOperator{\Ein}{Ein}

\newcommand{\uEin}{\widetilde{\Ein}}
\newcommand{\ie}{\emph{ie. }}

\newcommand{\sCa}[2]{\langle #1 | #2\rangle}         

\newcommand{\mcal}[1]{\ensuremath{\mathcal{#1}}}
\newcommand{\orth}{\bot}
\newcommand{\cY}{\mcal{Y}}

\newcommand{\cF}{\mcal{F}}
\newcommand{\cK}{\mcal{K}}

\newcommand{\cC}{\mcal{C}}
\newcommand{\cH}{\mcal{H}}

\begin{document}

\title[Quasi-Fuchsian AdS representations]
{Deformations of Fuchsian AdS representations are Quasi-Fuchsian}

\author[T. Barbot]{Thierry Barbot$^\dagger$}

\email{Thierry.Barbot@univ-avignon.fr}
\address{LMA, Avignon University\\
33, rue Louis Pasteur, 84 000 Avignon}



\begin{abstract}
Let $\Gamma$ be a finitely generated group, and let $\op{Rep}(\Gamma, \SO(2,n))$ be the moduli space
of representations of $\Gamma$ into $\SO(2,n)$ ($n \geq 2$).
An element $\rho: \Gamma \to \SO(2,n)$ of $\op{Rep}(\Gamma, \SO(2,n))$ is
\textit{quasi-Fuchsian} if it is faithful, discrete, preserves an
acausal subset in the conformal boundary $\Ein_n$ of the anti-de Sitter space; and if the associated
globally hyperbolic anti-de Sitter space is spatially compact - a particular
case is the case of \textit{Fuchsian representations}, \ie composition
of a faithfull, discrete and cocompact representation $\rho_f: \Gamma \to \SO(1,n)$ and
the inclusion $\SO(1,n) \subset \SO(2,n)$.

In \cite{merigot} we proved that quasi-Fuchsian representations are precisely
representations which are Anosov as defined in \cite{labourie}. In the present paper, we prove that 
the space of quasi-Fuchsian representations is open and closed, ie. that it
is an union of connected components of $\op{Rep}(\Gamma, \SO(2,n))$.

The proof involves the following fundamental result: let $\Gamma$ be the
fundamental group of a globally hyperbolic spatially compact spacetime locally modeled on $\AdS_n$,
and let $\rho: \Gamma \to \SO_0(2,n)$ be
the holonomy representation. Then, if $\Gamma$ is Gromov hyperbolic, the $\rho(\Gamma)$-invariant achronal limit set in $\Ein_n$ is acausal.

Finally, we also provide the following characterization of representations
with zero bounded Euler class: they are precisely the representations 
preserving a closed achronal subset of $\Ein_n$.

\end{abstract}

\maketitle

\tableofcontents

\section{Introduction}
Let $\SO_0(1,n)$, $\SO_0(2,n)$ denote the identity components of respectively
$\SO(1,n)$, $\SO(2,n)$ ($n \geq 2$).
Let $\Gamma$ be a cocompact torsion free lattice in $\SO_{0}(1,n)$.
For any Lie group $G$ we consider the moduli space of representations of $\Gamma$ into $G$
modulo conjugacy, equipped with the usual topology as an algebraic variety (see for example \cite{goldmillson}):
$$\op{Rep}(\Gamma, G) := \op{Hom}(\Gamma, G)/G$$

In the case $G=\SO_{0}(2,n)$ we distinguish the \textit{Fuchsian representations:} they are the representations
obtained by composition of the natural embedding $\SO_{0}(1,n) \subset \SO_{0}(2,n)$
and any faithful and discrete representation of $\Gamma$ into $\SO_{0}(1,n)$.
The space of faithful and discrete representations of $\Gamma$ into $\SO_{0}(1,n)$ is the union
of two connected components of $\op{Rep}(\Gamma, \SO_{0}(1,n))$: for $n\geq3,$ it follows from Mostow rigidity Theorem,
and for $n=2$, it follows from the connectedness of the Teichm\"uller space - observe
that there are indeed two connected components:  one corresponding to representations
such that $\rho^*\xi = \xi$, and the other to representations for which $\rho^*\xi=-\xi$, where $\xi$
is a generator of $\op{H}^n(\SO_0(1,n), \ZZ)$.

It follows that the space of Fuchsian representations is the union of two connected subsets of
$\op{Rep}(\Gamma, \SO_{0}(2,n)).$
Therefore, one can consider the union $\op{Rep}_{0}(\Gamma, G)$ of connected components  of
$\op{Rep}(\Gamma, \SO_{0}(2,n))$ containing
all the Fuchsian representations. The main result of the present paper is\footnote{This is a positive answer to Question 8.1 in \cite{merigot}.}:

\begin{theo}
\label{thm:main}
Every deformation of a Fuchsian representation, \ie every element of $\op{Rep}_{0}(\Gamma, \SO_0(2,n))$ is faithful and discrete.
\end{theo}

If one compares this result with the \textit{a priori} similar theory of deformations
of Fuchsian representations into $\SO_0(1,n+1)$, one observes that the situation is at first glance completely different: it is well-known that large deformations
of Fuchsian representations are \textbf{not} faithful and discrete; Fuchsian representations actually can be deformed to the trivial representation!

On the other hand, Theorem~\ref{thm:main} is very similar to the principal Theorem in~\cite{labourie} in the case
$G=\op{SL}(n, \RR)$, and where $\Gamma$ is a cocompact lattice in $\SO_0(1,2)$, \ie a closed surface group.
In this situation, Fuchsian representations are induced by the inclusion $\Gamma \subset \SO_0(1,2)$
and the morphism $\SO_0(1,2) \to \op{SL}(n, \RR)$ corresponding to the unique $n$-dimensional irreducible
representation of $\SO_0(1,2)$. The elements of
$\op{Rep}(\Gamma, \op{SL}(n, \RR))$ in the same connected component than the Fuchsian representations are called \textit{quasi-Fuchsian}. In \cite{labourie}, F. Labourie proves
that quasi-Fuchsian representations are \textit{hyperconvex}, \ie that they are faithfull, have discrete image, and preserve some curve in the projective space $\PP(\RR^{n})$ with some very strong
convexity properties (in particular, this curve is strictly convex). Later, O. Guichard proved in~\cite{guichard} that conversely hyperconvex representations
are quasi-Fuchsian.

At the very heart of the theory is the notion of
\textit{$(G,P)$-Anosov representation} (or simply Anosov
representation when there is no ambiguity about the pair $(G,P)$), where $G$
is a Lie group acting on any topological space $P$. The group
$\Gamma$ in general is a Gromov hyperbolic finitely generated group
(\cite{guichard3}; see also Sect.~8 in \cite{merigot}); typically, a closed surface group,
or, more generally, a cocompact lattice in $\SO_0(1,k)$ for some $k$.

Unfortunately, the terminology is not uniform in the literature. For example, what is called a
$(\SO_0(1,n+1), \partial\HH^{n+1})$-Anosov representation in \cite{guichard3} would
be called $(G, \cY)$-Anosov in the terminology of \cite{barflag} or \cite{merigot},
where $\cY$ is the space of spacelike geodesics of $\HH^{n+1}$.
We adopt here the definition and terminology used in \cite{guichard3}.

Simple, general
arguments ensure that Anosov representations are faithful, with
discrete image formed by loxodromic elements, and that they form an open domain
in $\op{Rep}(\Gamma, G)$. As a matter of fact, quasi-Fuchsian representations
into $\op{SL}(n, \RR)$ are $(\op{SL}(n, \RR), \cF)$-Anosov, where
$\cF$ is the frame variety\footnote{However, the converse is not necessarily true:
see~\cite{barflag} for the study of a family on non-hyperconvex $(\op{SL}(3, \RR), \cF)$-Anosov
representations.}.

The \textit{quasi-Fuchsian} terminology is inherited from hyperbolic geometry:
a representation $\rho: \Gamma \to \SO_0(1, n+1)$ is quasi-Fuchsian if it is faithfull, discrete, and preserves a topological $(n-1)$-sphere in $\partial\HH^{n+1}$. It is well-known by the experts that quasi-Fuchsian representations
into $\SO_0(1, n+1)$ are precisely the $(\SO_0(1, n+1), \partial\HH^{n+1})$-Anosov representations; and a proof can be obtained by adapting the arguments used in \cite{merigot}. It is also a direct consequence of Theorem~1.8 in \cite{guichard3}.

The anti de Sitter space $\AdS_{n+1}$ is the analog of the hyperbolic space $\HH^{n+1}$.
It is a Lorentzian manifold, of constant sectional curvature $-1$.
Whereas in the hyperbolic space pair of points are only distinguished by their mutual distance,
in the anti-de Sitter space we have to distinguish three types of pair
of points, according to the nature of the geodesic joining the two
points: this geodesic may be spacelike, lightlike or timelike --- in
the last two cases, the points are said \textit{causally related.}
Moreover, $\AdS_{n+1}$ is oriented, and admits also a \textit{time orientation,\/}
\ie an orientation of every nonspacelike geodesic. The group $\SO_0(2,n)$ is precisely the group of orientation and time orientation preserving isometries of $\AdS_{n+1}$.

The anti-de Sitter space $\AdS_{n+1}$ admits a conformal boundary
called the \textit{Einstein universe} and denoted by $\Ein_n$, which
plays a role similar to that of the conformal boundary $\partial
\HH^{n+1}$ for the hyperbolic space. The Einstein universe is a
conformal Lorentzian spacetime, and is also subject to a causality
notion: in particular, a subset $\Lambda$ of the Einstein space
$\Ein_n$ is called \emph{acausal} if any pair of distinct points in $\Lambda$
are the extremities of a spacelike geodesic in $\AdS_{n+1}$.

Once introduced these fundamental notions, we can state the main content of \cite{merigot}: let $\Gamma$ be a Gromov hyperbolic group.
For any representation $\rho: \Gamma \to \SO_0(2,n)$ the following notions coincide:

-- $\rho: \Gamma \to \SO_0(2,n)$ is $(\SO_0(2,n), \Ein_n)$-Anosov,

-- $\rho: \Gamma \to \SO_0(2,n)$ is faithful, discrete, and preserves an \textbf{acausal}
closed subset $\Lambda$ in the conformal boundary $\Ein_n$ of $\AdS_{n+1}$.

If furthermore $\Gamma$ is isomorphic to the fundamental group of a closed manifold of dimension $n,$ then $\Lambda$ is a topological $(n-1)$-sphere.

In particular, when $\Gamma$ is an uniform lattice in $\SO_0(1,n)$,
a representation of $\Gamma$ into $\SO_0(2,n)$ is called \textit{quasi-Fuchsian}
if it is faithful, discrete, and preserves an acausal topological $(n-1)$-sphere in $\Ein_n$. In other words, Theorem~\ref{thm:main} can be restated as follows: \textit{deformations (large or small) of Fuchsian representations into $\SO_0(2,n)$ are all quasi-Fuchsian.} It will be a corollary of the
following more general statement:

\begin{theo}
\label{thm:main2}
Let $n \geq 2$, and let $\Gamma$ be a Gromov hyperbolic group of cohomological dimension $\geq n$. Then, the moduli space $\op{Rep}_0(\Gamma, \SO_0(2,n))$ of $(\SO_0(2,n), \Ein_n)$-Anosov representations is open and closed in the modular space $\op{Rep}(\Gamma, \SO_0(2,n))$.
\end{theo}

\begin{remark}
\label{rk:cohom}
The reason for the hypothesis on the cohomological dimension is to ensure that the invariant closed acausal subset is a topological $(n-1)$-sphere.
It will follow from the proof that actually, under this hypothesis, if $\op{Rep}(\Gamma, \SO_0(2,n))$ is non-empty, then  $\Gamma$ is the fundamental group of a closed manifold,
and its cohomological dimension is precisely $n.$
\end{remark}

In order to present the ideas involved in the proof of Theorem~\ref{thm:main2} we need to
remind a bit further a few classical definitions in Lorentzian geometry.
By \emph{spacetime} we mean here an oriented Lorentzian manifold with a time orientation given by a smooth timelike vector field. This
allows to define the notion of future and past-directed causal curves.
A subset $\Lambda$ in $(M, g)$ is \textit{achronal} (respectively \textit{acausal}) if
there every timelike curve (respectively causal curve) joining two points in $\Lambda$ is
necessarily trivial, \ie reduced to one point. A \textit{time function} is a function
$t: M \to \RR$ which is strictly increasing along any causal curve.
A spacetime $(M,g)$ is \emph{globally hyperbolic spatially compact} (abbreviated to GHC)
if it admits a time function whose level sets are all compact.

Spatially compact global hyperbolicity is notoriously equivalent to the existence of
a \emph{compact Cauchy hypersurface}, that is a compact achronal set $S$
which intersects every inextendible timelike curve at exactly one
point. This set is then automatically a locally Lipschitz hypersurface
(see~\cite[Sect.~14, Lemma 29]{oneill}).

Observe that all these notions are not really associated to the Lorentzian metric
$g$, but to its conformal class $[g]$. Hence they are relevant to the Einstein universe,
which is naturally equipped with a $\SO_0(2,n)$-invariant conformal class of Lorentzian metric, but without any $\SO_0(2,n)$-invariant representative.

The key fact used in \cite{merigot} is that $(\SO_0(2,n), \Ein_n)$-Anosov
representations are holonomy representations of GHC spacetimes locally modeled
on $\AdS_{n+1}$. Thanks to the work of G. Mess and his followers (\cite{mess1, mess2}) the classification of GHC locally $\AdS$ spacetimes has been almost completed: they are in $1-1$ correspondance with \textit{GHC-regular representations.}

More precisely: let $\Gamma$ be a torsion-free finitely generated group of cohomological dimension $n$. A morphism $\rho: \Gamma \to \SO_0(2,n)$ is a GHC-regular representation if it is faithfull, discrete, and preserves an \textbf{achronal} closed $(n-1)$-topological sphere $\Lambda$ in $\Ein_n$. Define the invisible domain $E(\Lambda)$ as the domain in $\AdS_{n+1}$ comprising points that are not causally related to any element of $\Lambda$ (cf. Sect.~\ref{s:regulardomain}). The action of $\rho(\Gamma)$ on $E(\Lambda)$ is then free and properly discontinuous; the quotient space, denoted by $M_\rho(\Lambda)$, is GHC. Moreover, every maximal GHC spacetime locally modeled on $\AdS$ has this form. Also observe that $\Lambda$ only depends on $\rho$: there is at most one such invariant achronal sphere.
Finally, if the limit set $\Lambda$ is acausal, then the group $\Gamma$ is Gromov hyperbolic (actually, in this case, $\Gamma$ acts properly and cocompactly on a $\op{CAT}(-1)$ metric space, see Proposition 8.3 in \cite{merigot}).

Therefore, the only reason a GHC-regular representation may fail to be $(\SO_0(2,n), \Ein_n)$-Anosov is that the achronal sphere $\Lambda$ might be non acausal.
The main result of the present paper, from which Theorem~\ref{thm:main2} follows quite directly, is:

\begin{theo}[Theorem~\ref{thm:hyperbolicanosov2}]
\label{thm:hyperbolicanosov}
Let $\rho: \Gamma \to \SO_0(2,n)$ be a GHC-regular representation, where $\Gamma$ is a Gromov hyperbolic group. Then the achronal limit set $\Lambda$ is acausal, \ie $\rho$ is
$(\SO_0(2,n), \Ein_n)$-Anosov.
\end{theo}

Even if not logically relevant to the proofs in the present paper, we point out that there are
examples of GHC-regular representations with non-acausal limit set $\Lambda$.
Let us describe briefly in this introduction the family detailled in Sect.~\ref{s:split}:
let $(p,q)$ be a pair of positive integers
such that $p+q=n$, and let $\Gamma$ be a cocompact lattice of $\SO_{0}(1,p) \times \SO_{0}(1,q)$.
The natural inclusion of $\SO_{0}(1,p) \times \SO_{0}(1,q)$ into $\SO_{0}(2,n)$ arising
from the orthogonal splitting $\RR^{2,n}=\RR^{1,p} \oplus \RR^{1,p}$ induces a representation
$\rho: \Gamma \to \SO_0(2,n)$ which is GHC-regular, but where the invariant achronal limit set
$\Lambda$ is not acausal. The quotient space $M_\rho(\Lambda) := \rho(\Gamma)\backslash E(\Lambda)$ is a GHC spacetime, called a \textit{split $\AdS$ spacetime,} and the representation
is a \textit{split regular representation} (Definition~\ref{def:split}).

Finally, in the last section, we give another characterization of GHC-representations. There is a
fundamental bounded cohomology class $\xi$ in $\op{H}^2_{b}(\SO_0(2,n), \ZZ)$, the \textit{bounded Euler class}. It can be alternatively defined as the bounded cohomology class induced by the natural K\"ahler form $\omega$ of the symmetric $2n$-dimensional space $\cT_{2n} := \SO_0(2,n)/(\SO_0(2)\times\SO_0(n))$, or as the one associated to the central exact sequence:
$$1 \to \ZZ \to \widetilde{\SO}_0(2,n) \to \SO_0(2,n) \to 1$$
If $\rho: \Gamma \to \SO_0(2,n)$ is GH, the pull-back $\rho^*(\xi)$ (the \textit{Euler class} $\eu_b(\rho)$) is necessarily trivial. Actually:

\begin{theo}
\label{thm:euler}
Let $\rho: \Gamma \to \SO_0(2,n)$ be a faithful and discrete representation, where $\Gamma$ is the fundamental group of a negatively curved closed manifold $M$. The following assertions are equivalent:
\begin{enumerate}
\item $\rho$ is $(\SO_0(2,n), \Ein_n)$-Anosov,
\item the bounded Euler class $\eu_b(\rho)$ vanishes.
\end{enumerate}
\end{theo}

As a last comment, we recall part of the conjecture already proposed in \cite{merigot}[Conjecture $8.11$]: we expect
that GHC-regular representations of hyperbolic groups are all quasi-Fuchsians; in other words, that if a hyperbolic group $\Gamma$ admits a GHC-regular  representation into $\SO_0(2,n)$, then it must be isomorphic to a lattice in $\SO_0(1,n)$.

We expect actually a bit more. According to Theorem~\ref{thm:main2}, the space of GHC-regular representations is open and closed, hence an union of connected components of $\op{Rep}(\Gamma, \SO_0(2,n))$. 
It would be interesting to prove eventually that it coincides with $\op{Rep}_0(\Gamma, \SO_0(2,n))$, \ie that quasi-Fuchsian representations are all deformations of Fuchsian representations.

\subsection*{Acknowledgements} 
I would like to thanks A. Wienhard and O. Guichard for their encouragement to write the paper, and also F. Gu\'eritaud and F. Kassel for their interest, remarks and help.
O. Guichard also contributed to improve a first version of this paper.
This work has been supported by ANR grant GR-A-G (ANR-2011-BS01-003-02) and ANR grant ETTT (ANR-09-BLAN-0116-01).


\section{Preliminaries}
\label{sec.prelim}

We assume the reader sufficiently acquainted to basic causality notions in
Lorentzian manifolds like \emph{causal} or \emph{timelike} curves,
\emph{inextendible} causal curves, \textit{Lorentzian length} of causal curves, \emph{time orientation,\/}
\emph{future} and \emph{past} of subsets, \emph{time function,}
\emph{achronal} subsets, etc..., so that the brief description provided
in the introduction above is sufficient. We refer to~\cite{beem} or~\cite[section 14]{oneill} for further details.

\begin{defi}
A \textit{spacetime} is a connected, oriented, and time-oriented Lorentzian manifold.
\end{defi}

\subsection{Anti-de Sitter space}
Let $\RR^{2,n}$ be the vector space of dimension $n+2$, with coordinates
$(u,v, x_{1}, \ldots , x_{n})$, endowed with the quadratic
form:
\[ \mathrm{q}_{2,n}(u,v,x_1,\hdots,x_n) :=- u^{2} - v^{2} + x_{1}^{2} + \ldots + x_{n}^{2} \]

We denote by $\sCa{x}{y}$ the associated scalar product. For any
subset $A$ of $\RR^{2,n}$ we denote $A^{\orth}$ the orthogonal of $A$,
\ie the set of elements $y$ in $\RR^{2,n}$ such that $\sCa{y}{x}=0$
for every $x$ in $A$.  We also denote by $\cC_{n}$ the isotropic cone
$\{ w \in \RR^{2,n} /  \mathrm{q}_{2,n}(w)=0 \}$.

\begin{defi}
The anti-de Sitter space $\AdS_{n+1}$ is the hypersurface $\{ x \in
\RR^{2,n} /  \mathrm{q}_{2,n}(x)=-1 \}$ endowed with the Lorentzian
metric obtained by restriction of $\mathrm{q}_{2,n}$.
\end{defi}

At every element $x$ of $\AdS_{n+1}$, there is a canonical identification between the tangent space $T_x\AdS_{n+1}$ and the $\mathrm{q}_{2,n}$-orthogonal $x^\orth$

We will also consider the coordinates $(r, \theta, x_{1}, \hdots, x_{n})$ with:
\[ u=r\cos(\theta), v=r\sin(\theta)\]
We equip $\AdS_{n+1}$ with the time orientation defined
by this vector field, \ie the time orientation such that the timelike vector field $\frac\partial{\partial\theta}$ is everywhere
future oriented.

Observe the analogy with the definition of hyperbolic space $\HH^{n}$. Moreover, for every real number $\theta_0$, the subset $H_{\theta_0} := \{ (r, \theta, x_1,\hdots, x_n) /
\theta=\theta_{0} \} \subset \RR^{2,n}$ is a totally geodesic copy of
the hyperbolic space embedded in $\AdS_{n+1}$. More generally, the
totally geodesic subspaces of dimension $k$ in $\AdS_{n+1}$ are
connected components of the intersections of $\AdS_{n+1}$ with the
linear subspaces of dimension $(k+1)$ in $\RR^{2,n}$.

\begin{remark}
\label{rk:geodtheta}
In particular, geodesics are intersections with $2$-planes. Timelike geodesics
can all be described in the following way: let $x$, $y$ two elements of
$\AdS_{n+1}$ such that $\sCa{x}{y} = 0$. Then, when $\theta$ describes $\RR_{/2\pi\ZZ}$
the points $c(\theta) := \cos(\theta)x + \sin(\theta)y$ describe a future oriented timelike
geodesic containing $x$ (for $\theta = 0$) and $y$ (for $\theta = \pi/2$), parametrized by
unit length: the Lorentzian length of the restriction of $c$ to $(0, \theta)$ is $\theta$.
\end{remark}




\subsection{Conformal model }

\begin{prop}
\label{p.causal-structure}
The anti-de Sitter space $\AdS_{n+1}$ is conformally equivalent to
$(\SS^1\times\DD^{n},-d\theta^2+ds^2)$, where $d\theta^2$ is the standard Riemannian
metric on $\SS^1=\RR/2\pi\ZZ$, where $ds^2$ is the standard metric (of
curvature $+1$) on the sphere $\SS^{n}$ and $\DD^{n}$ is the open upper
hemisphere of $\SS^{n}$.
\end{prop}

\begin{proof}
In the $(r, \theta, x_{1}, ... , x_{n})$-coordinates the $\AdS$ metric is:
\[ -r^{2}\op{d\theta}^{2} + \op{ds}_{hyp}^{2} \]
where $\op{ds}_{hyp}^{2}$ is the hyperbolic metric, \ie the induced
metric on $H_0 = \{ (r,\theta, x_1,\hdots,x_n) /  \theta=0 \} \approx \HH^{n}$.  More precisely, $H_0$ is a sheet of the
hyperboloid $\{ (r, x_1, \hdots, x_n) \in \RR^{1,n} / -r^{2} +
x^{2}_{1} + ... + x^{2}_{n}=-1 \}$.  The map $(r, x_{1}, \ldots ,
x_{n}) \to (1/r, x_{1}/r, \ldots, x_{n}/r)$ sends this hyperboloid on $\DD^{n}$, and an easy computation shows that the pull-back by this
map of the standard metric on the hemisphere is
$r^{-2}\op{ds}_{hyp}^{2}$.  The proposition follows.
\end{proof}

Proposition~\ref{p.causal-structure} shows in particular that $\AdS_{n+1}$
contains many closed causal curves (including all timelike geodesics, cf. Remark~\ref{rk:geodtheta}).
But the universal covering $\wt\AdS_{n+1}$, conformally
equivalent to $(\RR\times\DD^{n},-d\theta^2+ds^2),$ contains no periodic causal curve.
It is strongly causal, but not globally hyperbolic (see Definition~\ref{def:gh}).

\subsection{Einstein universe}
\label{sub.einuniv}Einstein universe $\Ein_{n+1}$ is the product $\SS^{1} \times \SS^{n}$
endowed with the metric $-d\theta^{2} + ds^{2}$ where $ds^{2}$ is as above the standard spherical metric. The universal
Einstein universe $\wt\Ein_{n+1}$ is the cyclic covering $\RR \times \SS^{n}$ equipped with the
lifted metric still denoted $-d\theta^{2} + ds^{2}$, but where $\theta$ now takes value in $\RR$.
Observe that for $n\geq2$, $\wt\Ein_{n+1}$ is the universal covering, but it is not true for $n=1$.
According to this definition, $\Ein_{n+1}$ and $\wt\Ein_{n+1}$ are Lorentzian manifolds, but it is
more adequate to consider them as conformal Lorentzian manifolds.
We fix a time orientation: the one for which the coordinate $\theta$ is a time function on $\wt\Ein_{n+1}$.

In the sequel, we denote by $\mathrm{p}: \wt\Ein_{n+1} \to \Ein_{n+1}$ the
cyclic covering map. Let $\delta: \wt\Ein_{n+1} \to \wt\Ein_{n+1}$ be a generator of the Galois group of this
cyclic covering. More precisely, we select $\delta$ so that for any $\tilde{x}$ in
$\wt\Ein_{n+1}$ the image $\delta(\tilde{x})$
is in the future of $\tilde{x}$.

Even if Einstein universe is merely a conformal Lorentzian spacetime, one can define
the notion of \textit{photons,} \ie (non parameterized) lightlike geodesics. We can also
consider the causality relation in $\Ein_{n+1}$ and $\wt\Ein_{n+1}$. In particular, we define
for every $x$ in $\Ein_{n+1}$ the \textit{lightcone} $C(x)$: it is the union of photons containing
$x$. If we write $x$ as a pair $(\theta, \mathrm{x})$ in $\SS^1 \times \SS^n$, the lightcone $C(x)$ is the set
of pairs $(\theta', \mathrm{y})$ such that $|\theta'-\theta|=d(\mathrm{x},\mathrm{y})$ where $d$ is distance
function for the spherical metric $ds^{2}$.

There is only one point in $\SS^n$ at distance $\pi$ of $\mathrm{x}$: the antipodal point $-\mathrm{x}$.
Above this point, there is only one point in $\Ein_{n+1}$ contained in $C(x)$:
the antipodal point $-x=(\theta+\pi, -\mathrm{x})$. The lightcone $C(x)$ with the points
$x$, $-x$ removed is the union of two components:

-- the \textit{future cone:} it is the set $C^+(x):=\{ (\theta',\mathrm{y}) / \theta < \theta' < \theta+\pi, \; d(\mathrm{x},\mathrm{y})=\theta'-\theta \}$,

-- the \textit{past cone:} it is the set $C^-(x):=\{ (\theta',\mathrm{y}) / \theta-\pi < \theta' < \theta, \; d(\mathrm{x},\mathrm{y})=\theta-\theta' \}$.

Observe that the future cone of $x$ is the past cone of $-x$, and that the past cone
of $x$ is the future cone of $-x$.

According to Proposition~\ref{p.causal-structure} $\AdS_{n+1}$ (respectively $\wt\AdS_{n+1}$) conformally embeds in
$\Ein_{n+1}$ (respectively $\wt\Ein_{n+1}$). Observe that this embedding preserves the time orientation.
Since the boundary $\partial\DD^{n}$
is an equatorial sphere, the boundary $\partial\wt\AdS_{n+1}$ is a copy of the Einstein universe
$\wt\Ein_{n}$. In other words, one can attach a ``Penrose boundary'' $\partial\wt\AdS_{n+1}$ to
$\wt\AdS_{n+1}$ such that $\wt\AdS_{n+1}\cup\partial\wt\AdS_{n+1}$ is conformally
equivalent to $(\SS^1\times\overline{\DD}^{n},-d\theta^2+ds^2)$, where
$\overline{\DD}^{n}$ is the closed upper hemisphere of $\SS^{n}$.

The restrictions of $\mathrm{p}$ and $\delta$ to $\wt\AdS_{n+1} \subset \wt\Ein_{n+1}$
are respectively a covering map over $\AdS_{n+1}$ and a generator of the Galois group
of the covering; we will still denote them by $\mathrm{p}$ and $\delta$.

\subsection{Isometry groups}
Every element of $\SO(2,n)$ induces an isometry of $\AdS_{n+1}$, and, for $n \geq 2$,
every isometry of $\AdS_{n+1}$ comes from an element of $\SO(2,n)$. Similarly,
conformal transformations of $\Ein_{n+1}$ are projections of elements of $\SO(2,n+1)$ acting
on $\cC_{n+1}$ (still for $n \geq 2$).

In the sequel, we will only consider isometries preserving the orientation and the time orientation,
\ie elements of the neutral component $\SO_{0}(2,n)$ (or $\SO_{0}(2,n+1)$).

Let $\widetilde{\SO}_0(2,n)$ be the group of orientation and time orientation preserving isometries of $\wt\AdS_{n+1}$ (or conformal transformations of $\wt\Ein_n$). There is a central exact sequence:
$$1 \to \ZZ \to \widetilde{\SO}_0(2,n) \to \SO_0(2,n) \to 1$$
where the left term is generated by the transformation $\delta$ generating the Galois group
of $\op{p}: \wt\Ein_n \to \Ein_n$ defined previously.
Observe that for $n\geq3$, $\widetilde{\SO}_0(2,n)$ is the universal covering of $\SO_0(2,n)$.

\subsection{Achronal subsets}
\label{sub.achro}
Recall that a subset of a conformal Lorentzian manifold is \textit{achronal} (respectively \textit{acausal}) if there is
no timelike (respectively causal) curve joining two distinct points of the subset.
In $\Ein_{n}\approx
(\RR\times\SS^{n-1},-d\theta^2+ds^2)$, every achronal subset is precisely the graph
of a $1$-Lipschitz function $f: \Lambda_0 \rightarrow {\mathbb R}$ where
$\Lambda_0$ is a subset of ${\mathbb S}^{n-1}$ endowed with its canonical
metric $d$. In particular,
the achronal closed topological hypersurfaces in $\partial\wt\AdS_{n+1}$ are
exactly the graphs of the $1$-Lipschitz functions $f:\SS^{n-1}\to\RR$: they are
topological $(n-1)$-spheres.

Similarly, achronal subsets of $\wt\AdS_{n+1}$ are graphs of $1$-Lipschitz functions
 $f: \Lambda_0 \rightarrow {\mathbb R}$ where $\Lambda_0$ is a subset of ${\mathbb D}^{n}$,
 and achronal topological hypersurfaces are graphs of $1$-Lipschitz maps $f: \DD^{n} \rightarrow {\mathbb R}$.

\emph{Stricto-sensu,\/} there is no achronal subset in $\Ein_{n+1}$ since closed timelike curves
through a given point cover the entire
$\Ein_{n+1}$. Nevertheless, we can keep track of this notion in $\Ein_{n+1}$ by defining ``achronal'' subsets
of $\Ein_{n+1}$ as projections of genuine achronal subsets of $\wt\Ein_{n+1}$. This definition is
justified by the following results:

\begin{lem}[Lemma 2.4 in \cite{merigot}]
\label{le.achrinj}
The restriction of $\mathrm{p}$ to any achronal subset of $\wt\Ein_{n+1}$ is injective.\fin
\end{lem}

\begin{cor}[Corollary 2.5 in \cite{merigot}]
\label{cor.lift}
Let $\wt\Lambda_{1}$, $\wt\Lambda_{2}$ be two achronal subsets of $\wt\Ein_{n+1}$
admitting the same projection in $\Ein_{n+1}$. Then there is an integer $k$ such that:
$$\wt\Lambda_{1}=\delta^{k}\wt\Lambda_{2}$$
where $\delta$ is the generator of the Galois group introduced above.
\fin
\end{cor}

\subsection{The Klein model $\ADS_{n+1}$ of the anti-de Sitter
  space}
We now consider the quotient $\SS(\RR^{2,n})$ of
$\RR^{2,n}\setminus\{0\}$ by positive homotheties. In other
words, $\SS(\RR^{2,n})$ is the double covering of the projective space
$\PP(\RR^{2,n})$. We denote by $\SS$ the projection of
$\RR^{2,n}\setminus\{0\}$ on $\SS(\RR^{2,n})$.
For every $\op{x}$, $\op{y}$ in $\SS(\RR^{2,n})$, we denote by $\langle \op{x} \mid \op{y} \rangle$
the \textbf{sign} of the real number $\langle x \mid y \rangle$, where $x,y\in\RR^{2,n}$
are representatives of $\op{x}$, $\op{y}$.
The \emph{Klein model} $\ADS_{n+1}$ of the
anti-de Sitter space is the projection of $\AdS_{n+1}$ in $\SS(\RR^{2,n})$,
endowed with the induced Lorentzian metric, \ie:
$$\ADS_{n+1} := \{ \op{x} \in \SS(\RR^{2,n}) \;/\; \langle \op{x} \mid \op{x} \rangle < 0 \}$$

The topological boundary of $\ADS_{n+1}$ in $\SS(\RR^{2,n})$ is the
projection of the isotropic cone $\cC_n$; we will denote this
boundary by $\partial\ADS_{n+1}$. The projection $\SS$
defines an one-to-one isometry between $\AdS_{n+1}$ and $\ADS_{n+1}$.
The continuous extension of this isometry is a canonical
homeomorphism between
$\AdS_{n+1}\cup\partial\AdS_{n+1}$ and $\ADS_{n+1}\cup\partial\ADS_{n+1}$.

For every linear subspace $F$ of dimension $k+1$ in $\RR^{2,n}$, we
denote by $\SS(F):=\SS(F\setminus\{0\})$ the corresponding projective subspace of
dimension $k$ in $\SS(\RR^{2,n})$. The geodesics of $\ADS_{n+1}$ are the
connected components of the intersections of $\ADS_{n+1}$ with the
projective lines $\SS(F)$ of $\SS(\RR^{2,n})$. More generally, the
totally geodesic subspaces of dimension $k$ in $\ADS_{n+1}$ are the
connected components of the intersections of $\ADS_{n+1}$ with the
projective subspaces $\SS(F)$ of dimension $k$ of $\SS(\RR^{2,n})$.

\begin{defi}
\label{def.affine}
For every $\op{x} = \SS(x)$ in $\ADS_{n+1}$, we define the \emph{affine domain} (also denoted by $U(x)$):
$$U(\op{x}) := \{ \op{y} \in \ADS_{n+1} \;/\; \langle \op{x} \mid \op{y} \rangle < 0 \}$$

In other words, $U(\op{x})$ is the connected component of
$\ADS_{n+1}\setminus\SS(x^{\orth})$ containing $\op{x}$. Let $V(\op{x})$ (also denoted by $V(x)$)
be the connected component of
$\SS(\RR^{2,n})\setminus\SS(x^{\orth})$ containing $U(\op{x})$. The boundary $\partial
U(\op{x}) \subset\partial\ADS_{n+1}$ of $U(\op{x})$ in $V(\op{x})$ is called the \textit{affine
  boundary\/} of $U(\op{x})$.
\end{defi}

\begin{remark}
\label{rk.affine}
Up to composition by an element of the isometry group $SO_0(2,n)$ of
$\mathrm{q}_{2,n}$, we can assume that $\SS(x^{\orth})$ is the projection of the
hyperplane $\{ u=0\} $ in $\RR^{2,n}$ and $V(x)$ is the projection of the
region $\{ u>0\} $ in $\RR^{2,n}$. The map
$$(u, v, x_1,x_2,\dots,x_{n+1})\mapsto (t, \bar{x}_1,\dots,
\bar{x}_{n}) := \left(\frac{v}{u},
\frac{x_1}{u},\frac{x_2}{u},\dots,\frac{x_{n}}{u}\right)$$ induces a
diffeomorphism between $V(x)$ and $\RR^{n+1}$ mapping the affine
domain $U(x)$ to the region $\{ (t, \bar{x}_1,\hdots,\bar{x}_n) \in
\RR^{n+1}\vert\, \mathrm{q}_{1,n}(t,\bar{x}_1,\hdots,\bar{x}_n) < 1\}
$, where $\mathrm{q}_{1,n}$ is the Minkowski norm. The affine boundary
$\partial U(x)$ corresponds to the hyperboloid $\{
(t,\bar{x}_1,\hdots,\bar{x}_n \vert\, \mathrm{q}_{1,n}(t,\bar{x}_1,\hdots,\bar{x}_n)
=1\} $. The intersections between $U(x)$ and the
totally geodesic subspaces of $\ADS_{n+1}$ correspond to the
intersections of the region $\{ (t, \bar{x}_1,\hdots,\bar{x}_n) \in
\RR^{n+1}\vert\, \mathrm{q}_{1,n}(t,\bar{x}_1,\hdots,\bar{x}_n) < 1\}
$ with the affine subspaces of $\RR^{n+1}$.
\end{remark}

\begin{lemma}[Lemma 10.13 in \cite{ABBZ}]
\label{l.causally-related}
Let $U$ be an affine domain in $\ADS_{n+1}$ and $\partial
U\subset\partial\ADS_{n+1}$ be its affine boundary. Let $\op{x}$ be be a point
in $\partial U$, and $\op{y}$ be a point in $U\cup\partial U$. There exists
a causal (resp. timelike) curve joining $\op{x}$ to $\op{y}$ in $U\cup\partial
U$ if and only if $\langle \op{x}\mid \op{y} \rangle \geq 0$ (resp. $\langle
\op{x}\mid \op{y} \rangle>0$).\fin
\end{lemma}

\begin{remark}
\label{rk:pastcomponent}
The boundary of $U(\op{x})$ in $\ADS_{n+1}$ is $\SS(x^{\orth}) \cap \ADS_{n+1}$.
It has two boundary components: the \textit{past component} $H^-(\op{x})$
and the \textit{future component} $H^+(\op{x})$. These components are characterized
by the following property: timelike geodesics enter in $U(\op{x})$ through $H^-(\op{x})$
and exit through $H^+(\op{x})$.

They call also be defined as follows: let $\wt{U}(\op{x})$ be a lifting in $\wt\AdS_{n+1}$ of
$U(\op{x})$, and let $\wt H^\pm(\op{x})$ be the lifts of $H^\pm(\op{x})$.
Then, $\wt{U}(\op{x})$ is the intersection between the future of $H^-(\op{x})$
and the past of $H^+(\op{x})$.

The boundary components $H^\pm(\op{x})$ are totally geodesic
embedded copies of $\HH^n$. They are also called \textit{hyperplanes dual to $\op{x}$,} and we distinguish the hyperplane past-dual $H^-(\op{x}) = H^-(\op{x})$ from the hyperplane future-dual $H^+(\op{x})=H^+(\op{x})$.

Last but not least: $H^\pm(\op{x})$ have also the following characteristic property: every future oriented (resp. past oriented) timelike geodesic starting at $\op{x}$ reach $H^+(\op{x})$ (resp. $H^-(\op{x})$) at time $\pi/2$ (see Remark~\ref{rk:geodtheta}). In other words, $H^\pm(\op{x})$ is the set of points at Lorentzian distance $\pm\pi/2$ from $\op{x}$.
\end{remark}

\subsection{The Klein model of the Einstein universe}
Similarly, Einstein universe has a Klein model:
the projection $\SS(\cC_n)$ in $\SS(\RR^{2,n})$ of the
isotropic cone $\cC_n$ in $\RR^{2,n}$.
The conformal Lorentzian structure can be defined in terms
of the quadratic form $\mathrm{q}_{2,n}$ (for more details, see~\cite{franceseinstein, primer}).

\begin{remark}
In the sequel, we will frequently identify $\Ein_{n}$ with $\SS(\cC_n)$,
since we will frequently switch from one model to the other.
\end{remark}

An immediate
corollary of Lemma~\ref{l.causally-related} is:

\begin{cor}
\label{cor.list}
For $\Lambda \subseteq \Ein_n$, the following assertions are
equivalent.
\begin{enumerate}
\item $\Lambda$ is achronal (respectively acausal);
\item when we see $\Lambda$ as a subset of $\SS(\cC_n) \approx \Ein_n$ the
scalar product $\langle \op{x} \mid \op{y} \rangle$ is non-positive
(respectively negative) for every distinct $\op{x}, \op{y}\in\Lambda$.
\end{enumerate}\fin
\end{cor}

\begin{remark}
\label{rk:minkconf}
Let $\op{x}_0$ be any element of $\Ein_n \approx \SS(\cC_n)$. Then, the open domain defined by:
$$\op{Mink}(\op{x}_0) = \{ \op{x} \in \SS(\cC_n) \;/\; \langle \op{x}_0 \mid \op{x} \rangle < 0\}$$
is conformally isometric to the Minkowski space $\RR^{1,n-1}$ (see~\cite{franceseinstein, primer}).

In particular, the stabilizer $G_0$ of $\op{x}_0$ in $\SO_0(2,n)$ is isomorphic to
the group of conformal isometries of $\RR^{1,n-1}$, \ie of affine transformations
whose linear part has the form $x \mapsto \lambda g(x)$, where $\lambda$ is a positive real number
and $g$ an element of $\SO_0(1,n-1)$.
\end{remark}






\section{Regular $\AdS$ manifolds}
\label{sec.regads}

In all this section, $\wt\Lambda$ is a closed achronal subset of $\partial\wt\AdS_{n+1}$, and
$\Lambda$ is the projection of $\wt\Lambda$ in $\partial\AdS_{n+1}$.

\subsection{AdS regular domains}
\label{s:regulardomain}

We denote by  $\wt E(\wt\Lambda)$ the \emph{invisible domain} of $\wt\Lambda$ in
$\wt\AdS_{n+1}$, that is,

$$\wt E(\wt\Lambda)=:\wt\AdS_{n+1} \setminus
\left(J^-(\wt\Lambda)\cup J^+(\wt\Lambda)\right)
$$
where $J^-(\wt\Lambda)$ and $J^+(\wt\Lambda)$ are the causal past and the
causal future of $\wt\Lambda$ in
$\wt\AdS_{n+1}\cup\partial\wt\AdS_{n+1}=(\RR\times\overline{\DD}^{n-1},-d\theta^2+ds^2)$.
We denote by $\Cl(\wt E(\wt\Lambda))$ the closure of $\wt E(\wt\Lambda)$ in
$\wt\AdS_{n+1}\cup\partial\wt\AdS_{n+1}$ and by $E(\Lambda)$ the projection
of $\wt E(\wt \Lambda)$ in $\AdS_{n+1}$
(according to Corollary~\ref{cor.lift}, $E(\Lambda)$ only depends on $\Lambda$, not on
the choice of the lifting $\wt\Lambda$).

\begin{defi}
A $(n+1)$-dimensional \emph{AdS regular domain} is a domain of the form
$E(\Lambda)$ where $\Lambda$ is the projection in $\partial\AdS_{n+1}$ of an
achronal subset $\wt\Lambda\subset\partial\wt{\mbox{AdS}_{n+1}}$ containing
at least two points.
If $\wt\Lambda$ is a topological $(n-1)$-sphere, then $E(\Lambda)$ is
\emph{GH-regular} (this definition is motivated by Theorem~\ref{teo.cosmogood} and Proposition~\ref{pro.adsregular}).
\end{defi}

\begin{rema}
The invisible domain  $\wt E(\wt \Lambda)$ is causally convex in
of $\wt\AdS_{n+1}$; \ie every causal curve joining two points
in $\wt E(\wt \Lambda)$ is entirely contained in $\wt E(\wt \Lambda)$. This is an immediate consequence of the definitions. It follows that AdS regular domains are strongly causal.
\end{rema}


\begin{remark}
\label{r.f-f+}
Recall that $\wt\Lambda$ is the graph of a $1$-Lipschitz function
$f:\Lambda_0\to\RR$ where $\Lambda_0$ is a closed subset of
$\SS^{n-1}$ (section \ref{sub.achro}).  Define two functions
$f^{-},f^{+}:\overline{\DD}^{n}\to\RR$ as follows:
\begin{align*}
f^{-}(\mathrm{x}) &:= \mbox{Sup}_{\mathrm{y} \in \Lambda_0} \{ f(\mathrm{y})-d(\mathrm{x},\mathrm{y}) \} , \\
f^{+}(\mathrm{x}) &:= \mbox{Inf}_{\mathrm{y} \in \Lambda_0} \{ f(\mathrm{y})+d(\mathrm{x},\mathrm{y}) \} ,
\end{align*}
where $d$ is the distance induced by $\d s^2$ on
$\overline{\DD}^{n}$. It is easy to check that
$$
\wt E(\wt\Lambda)=\{(\theta, \mathrm{x})\in\RR\times{\DD}^{n} \mid
f^{-}(\mathrm{x})< \theta <f^{+}(\mathrm{x})\}.
$$
\end{remark}

\begin{remark}
\label{rk:extension}
Keeping the notations in the previous remark, observe that the graph of the restriction of $f^+$ (or $f^-$)
to $\partial\DD^n$ is a closed achronal $(n-1)$-sphere $\wt\Lambda^+$ (or $\wt\Lambda^-$) in $\wt\AdS_{n+1}$
which contains the initial achronal subset $\wt\Lambda$. They project to achronal $(n-1)$-spheres $\Lambda^\pm$ in $\partial\AdS_{n+1}$
that contain $\Lambda.$ 

Furthermore, any element $g$ of $\SO_0(2,n)$ preserving $\Lambda$ must preserve $E(\Lambda),$ hence the graphs of $f^\pm,$
and therefore must preserve $\Lambda^+$ and $\Lambda^-$.
\end{remark}

\begin{defi}
\label{def:horizon}
The graph of $f^-$ (respectively $f^+$) is a closed achronal subset of $\wt\AdS_{n+1}$, called the \textit{lifted past}
(respectively \textit{future}) \textit{horizon} of $\wt E(\wt\Lambda)$, and denoted $\cH^-(\wt\Lambda)$ (respectively $\cH^+(\wt\Lambda)$).

The projections in $\AdS_{n+1}$ of $\wt\cH^\pm(\wt\Lambda)$ are called \textit{past} and {future horizons} of $E(\Lambda)$, and
denoted $\cH^\pm(\Lambda)$.
\end{defi}

The following lemma is a refinement of Lemma~\ref{le.achrinj}:

\begin{lem}[Corollary 10.6 in \cite{ABBZ}]
\label{le.one-to-one}
For every (non-empty) closed achronal set
$\wt\Lambda\subset\partial\wt\AdS_{n+1}$, the projection of $\wt
E(\wt\Lambda)$ on $E(\Lambda)$ is one-to-one.\fin
\end{lem}

\begin{defi}
\label{def.purelightlike}
$\wt\Lambda$ is \emph{purely lightlike} if the associated subset
$\Lambda_0$ of $\SS^{n}$ contains two antipodal points $\mathrm{x}_0$ and $-\mathrm{x}_0$
such that, for the associated 1-Lipschitz map $f:\Lambda_0\to\RR$ the equality $f(\mathrm{x}_0)=f(-\mathrm{x}_0) +\pi$ holds.
\end{defi}

If $\wt\Lambda$ is purely lightlike, for every element $\mathrm{x}$ of
$\overline{\DD}^{n}$  we have $f^{-}(\mathrm{x})=f^{+}(\mathrm{x}) =
f(-\mathrm{x}_0) + d(-\mathrm{x}_0, \mathrm{x})=f(\mathrm{x}_0) - d(\mathrm{x}_0, \mathrm{x})$,
implying that $\wt E(\wt\Lambda)$ is empty.
Conversely:

\begin{lem}[Lemma 3.6 in \cite{merigot}]
\label{le.purevide}
$\wt E(\wt\Lambda)$ is empty if and only if $\wt\Lambda$ is purely lightlike. More precisely,
if for some point $\mathrm{x}$ in $\DD^{n}$ the equality $f^{+}(\mathrm{x})=f^{-}(\mathrm{x})$ holds then
$\wt\Lambda$ is purely lightlike.\fin
\end{lem}

\subsection{AdS regular domains as subsets of $\ADS_{n+1}$}
The canonical homeomorphism between $\AdS_{n+1}\cup\partial\AdS_{n+1}$ and
$\ADS_{n+1}\cup\partial\ADS_{n+1}$ allows us to see AdS regular domains as
subsets of $\ADS_{n+1}$.

Putting together the definition of the invisible domain $E(\Lambda)$
of a set $\Lambda\subset\partial\AdS_{n+1}$ and
Lemma~\ref{l.causally-related}, one gets:

\begin{prop}[Proposition 10.14 in \cite{ABBZ}]
\label{pro.proj}
If we see $\Lambda$ and $E(\Lambda)$ in the Klein
model $\ADS_{n+1}\cup\partial\ADS_{n+1}$, then
$$
E(\Lambda)=\{\op{y} \in \ADS_{n+1}\mbox{ such that }\langle
\op{y}  \mid \op{x} \rangle < 0 \mbox{ for every }\op{x}\in\Lambda\}
$$\fin
\end{prop}

\subsection{Convex core of AdS regular domains}
\label{sub.convexhull}

In this section, we assume that $\Lambda$ is not purely lightlike and not reduced to a single point.
The following notions are classical and well-known:

\begin{defi}
A subset $\Omega$ of $\SS(\RR^{2,n})$ is \textit{convex} if there is a convex cone $J$
of $\RR^{2,n}$ such that $\Omega = \SS(J)$. The \textit{relative interior} of $\Omega$, denoted by
$\Omega^\circ$ is the convex subset $\SS(J^\circ)$ where $J^\circ$ is the interior of $J$ in the subspace spanned by $J$.
\end{defi}

It is well-known that the closure of a convex subset is still convex, and that it coincides with the closure of the relative interior.

\begin{theodefi}
Let $\Omega = \SS(J)$ be a convex subset of $\SS(\RR^{2,n})$. The following assertions are equivalent:
\begin{itemize}
  \item $J$ contains no complete affine line,
  \item there is an affine hyperplane $H$ in $\SS(\RR^{2,n})$ such that $H \cap J$ is relatively compact in $H$ and such that $\Omega = \SS(J \cap H)$,
  \item The closure of $\Omega$ contains no pair of opposite points.
\end{itemize}
If one of these equivalent properties hold, then $\Omega$
is \textit{salient}.\fin
\end{theodefi}

\begin{defi}
Let $\Omega = \SS(J)$ a convex subset of $\SS(\RR^{2,n})$. The \text{dual of $\Omega$} is the closed
convex subset $\SS(J^\ast \setminus \{ 0 \})$ where:
$$J^\ast = \{ x \in \RR^{2,n} \;/\; \forall y \in J, \; \langle x \mid y \rangle \leq 0 \}$$
\end{defi}

\begin{prop}
\label{pro:dualdual}
Let $\Omega$ be a convex subset of $\SS(\RR^{2,n})$. Then, the bidual $\Omega^{\ast\ast}$ is the
closure $\Cl\left(\Omega\right)$ of $\Omega$ in $\SS(\RR^{2,n})$. The relative interior $\Omega^\circ$ is open in $\SS(\RR^{2,n})$ if and only if $\Omega^\ast$ is salient.\fin
\end{prop}

Let $\hat{\Lambda}$ be the preimage of $\Lambda \subset \Ein_n = \SS(\cC_n)$ by $\SS$.
The convex hull of $\hat{\Lambda}$ is a convex cone $\op{Conv}(\hat{\Lambda})$ in $\RR^{2,n}$, whose projection is a compact convex subset of $\SS(\RR^{2,n})$, denoted by $\op{Conv}(\Lambda)$, and called \textit{the convex hull of $\Lambda$}
and the \textit{convex core of $E(\Lambda)$.}

\begin{lemma}
The intersection between $\op{Conv}(\Lambda)$ and $\Ein_n$ is the union of lightlike segments in $\Ein_n$ joining two elements of $\Lambda$.
The relative interior $\op{Conv}(\Lambda)^\circ$ is contained in $\ADS_{n+1}$.
\end{lemma}

\begin{proof}
Elements of $\op{Conv}(\hat{\Lambda})$ are linear combinations $x  =\sum_{i=1}^{k} t_ix_i$ where $t_i$ are non-negative real numbers and $x_i$ elements of $\hat{\Lambda}$.
\begin{eqnarray*}
  \mathrm{q}_{2,n}(x) &=& \sum_{i, j = 1}^{k} t_it_j\langle x_i \mid x_j \rangle
\end{eqnarray*}
Since every $\langle x_i \mid x_j \rangle$ is nonpositive, we have $\mathrm{q}_{2,n}(x) \leq 0$.

Moreover, if $\mathrm{q}_{2,n}(x) = 0$, then every $\langle x_i \mid x_j \rangle$ must be equal to $0$, \ie the vector space spanned by the $x_i$'s is isotropic, hence either a line, or an
isotropic plane in $\cC_n$. In the first case, $x$ is an element of $\Lambda$, and in the second case, $x$ lies on a lighlike geodesic of $\Ein_n$ joining two elements of $\Lambda$.

Finally, assume that $\op{Conv}(\Lambda)^\circ$ is not contained in $\ADS_{n+1}$. Since
$\mathrm{q}_{2,n}(x) \leq 0$ for every $x$ in $\hat{\Lambda}$, it follows that $\op{Conv}(\hat{\Lambda})$
is contained in $\cC_n$, and more precisely, by the argument above, in an istropic 2-plane.
It is a contradiction since $\Lambda$ by hypothesis is not purely lightlike.
\end{proof}

Actually, the case where $\op{Conv}(\Lambda)^\circ$ is not an open subset of $\AdS_{n+1}$ is exceptional:

\begin{lemma}[Lemma 3.13 in \cite{merigot}]
\label{le.convempty}
If $\op{Conv}(\Lambda) \cap \AdS_{n+1}$ has empty interior, then it is contained in a totally geodesic  spacelike hypersurface of $\AdS_{n+1}$.\fin
\end{lemma}


Proposition~\ref{pro.proj} can be rewritten as follows:

\begin{prop}[Proposition 10.17 in \cite{ABBZ}]
\label{lem.adsdual}
The domain $E(\Lambda)$ is the intersection $\ADS_{n+1} \cap (\op{Conv}(\Lambda)^\ast)^\circ$. \fin
\end{prop}

\begin{remark}
\label{rk.nice}
A corollary of Proposition~\ref{lem.adsdual} is that the invisible
domain $E(\Lambda)$ is convex, hence contains $\op{Conv}(\Lambda)^\circ$.
\end{remark}

Hence, if $x$ lies in the interior of $\op{Conv}(\Lambda)$, the affine domain $U(x)$ contains the closure of $E(\Lambda)$. Therefore:

\begin{prop}
\label{pro:closure11}
Assume that $\Lambda$ is not the boundary of a totally geodesic copy of $\HH^n$ in $\AdS_{n+1}$. Then, the restriction of
$\hat{p}: \wt\AdS_{n+1} \to \AdS_{n+1}$ to the closure of $\wt{E}(\wt\Lambda)$ is one-to-one.

In particular, $\hat{p}: \wt\cH^\pm(\wt\Lambda) \to \cH^\pm(\Lambda)$ is injective.\fin
\end{prop}

The boundary of $E(\Lambda)$ in $\AdS_{n+1}$ has two components: the past and future horizons $\cH^\pm(\Lambda)$ (cf. Definition~\ref{def:horizon}).
Since $E(\Lambda)$ is convex, every point $x$ in $\cH^-(\Lambda)$ lies in a support hyperplane for $E(\Lambda)$, \ie a totally geodesic
hyperplane $H$ tangent to $\cH^-(\Lambda)$ at $x$. 
According to Proposition~\ref{lem.adsdual},
$H$ is the hyperplane dual to an element $p$ of $\partial\op{Conv}(\Lambda)$, hence
$H$ is either spacelike (if $p \in \ADS_{n+1}$) or degenerate (if $p \in \Ein_n$).

\begin{remark}
\label{rk:filling}
For every achronal subset $\Lambda$, the intersection $\op{Conv}(\Lambda) \cap \Ein_n$, which is an union of lightlike geodesic segments joining elements of $\Lambda$ is still achronal (since
$\sCa{\sum s_i x_i}{\sum t_jy_j} = \sum s_it_j\sCa{x_i}{y_j} \leq 0$ for $s_i, t_j \geq 0$, $x_i$, $y_j \in \Lambda$). We call it the \textit{filling of $\Lambda$} and denote it by $\op{Fill}(\Lambda)$. According to Proposition~\ref{lem.adsdual}:
$$E(\op{Fill}(\Lambda)) = E(\Lambda)$$
Hence, we can always assume wlog that $\Lambda$ is \textit{filled}, \ie $\Lambda = \op{Fill}(\Lambda)$.
\end{remark}

\section{Globally hyperbolic AdS spacetimes}
\label{sec.ghads}

In all this section, \textit{$\Lambda$ is a non-purely lightlike topological achronal $(n-1)$-sphere in $\partial\AdS_{n+1}$.} In particular, $\Lambda$ is automatically filled (cf. Remark~\ref{rk:filling}).

\begin{prop}[Corollary 10.7 in \cite{ABBZ}]
\label{pro.inter-boundary}
For every achronal topological $(n-1)$-sphere
$\Lambda\subset\partial\AdS_{n+1}$, the intersection between the closure $Cl\left(E(\Lambda)\right)$  of $E(\Lambda)$ in $\Ein_{n+1}$
and $\Ein_n = \partial\AdS_{n+1}$ is reduced to $\Lambda$.\fin
\end{prop}

The meaning of Proposition~\ref{pro.inter-boundary} is that $(\op{Conv}(\Lambda)^\ast)^\circ$ is already contained in $\AdS_{n+1}$, so that
the expression $E(\Lambda) = \ADS_{n+1} \cap (\op{Conv}(\Lambda)^\ast)^\circ$
is reduced to $E(\Lambda) = (\op{Conv}(\Lambda)^\ast)^\circ$ when $\Lambda$ is a topological sphere.

\begin{remark}
It follows from Proposition~\ref{pro.inter-boundary} that the GH-regular domain
$E(\Lambda)$ characterizes $\Lambda$, \ie invisible domains
of different achronal $(n-1)$-spheres are different. We call
$\Lambda$ the \textit{limit set of $E(\Lambda)$.}
\end{remark}

\subsection{More on the convex hull of achronal topological $(n-1)$-spheres}
\label{sec:moreconvex}
Recall that there are two maps $f^-$, $f^+$ such that
$\wt{E}(\wt{\Lambda})=\{ (\theta, \mathrm{x}) / f^-(\mathrm{x}) < \theta < f^+(\mathrm{x}) \}$
(cf. Definition~\ref{r.f-f+}).

\begin{prop}
\label{pro.F-F+}
The complement of $\Lambda$ in the boundary $\partial\op{Conv}(\Lambda)$ has
two connected components. Both are closed edgeless achronal subsets of $\AdS_{n+1}$. More precisely,
in the conformal model their liftings in $\wt{\AdS}_{n+1}$ are graphs
of 1-Lipschitz maps $F^+$, $F^-$ from $\DD^n$
into $\RR$ such that
\begin{equation}
\label{eq:fF}
f^- \leq F^- \leq F^+ \leq f^+
\end{equation}
\end{prop}

\begin{proof}
See Proposition 3.14 in \cite{merigot}. Observe that in \cite{merigot}, Proposition 3.14 is proved in the case where $\Lambda$ is acausal, and not Fuchsian
(the Fuchsian case being the case where $\Lambda$ is the boundary of a totally geodesic hypersurface in $\wt{\AdS}_{n+1}$).
Inequalities in equation (\ref{eq:fF}) are then all strict inequalities, which is false in the general case, as we will see later\footnote{Anyway, one can already observe that in the Fuchsian case $F^-=F^+$.}.
Nevertheless, the proof of Proposition 3.14 in \cite{merigot} can easily adapted, providing
a proof of Proposition~\ref{pro.F-F+}.
\end{proof}

We have already observed that $\partial E(\Lambda) \setminus \Lambda$ is the union of two achronal connected components $\cH^\pm(\Lambda)$; in a similar
way, $\partial\op{Conv}(\Lambda) \setminus \Lambda$ is the union of two achronal $n$-dimensional topological disks: the \textit{past component} $S^-(\Lambda)$ (the graph of $F^-$)
and the \textit{future component} $S^+(\Lambda)$. Since $E(\Lambda)$ and $\op{Conv}(\Lambda)$ are convex and dual one to the other, for every element $x$ in $S^-(\Lambda)$
(respectively $S^+(\Lambda)$) there is an element $p$ of $\Lambda$ or $\cH^+(\Lambda)$ (respectively $\cH^+(\Lambda)$) such that $H^-(p)$ (respectively $H^+(p)$) is a support hyperplane for $S^-(\Lambda)$
(respectively $S^+(\Lambda)$) at $x$: these support hyperplanes are either totally geodesic copies of $\HH^n$ (if $p \in \AdS_{n+1}$) or degenerate (if $p \in \Lambda$).

Similarly, at every element $x$ of $\cH^-(\Lambda)$ (respectively $\cH^+(\Lambda)$) there is a support hyperplane $H^-(p)$ (respectively $H^+(p)$) where $p$ is an element of
$S^+(\Lambda) \cup \Lambda$ (respectively $S^-(\Lambda) \cup \Lambda$) (see Figure~\ref{fig:convexhull}).

 \begin{figure}[ht]
  \centering
  \includegraphics{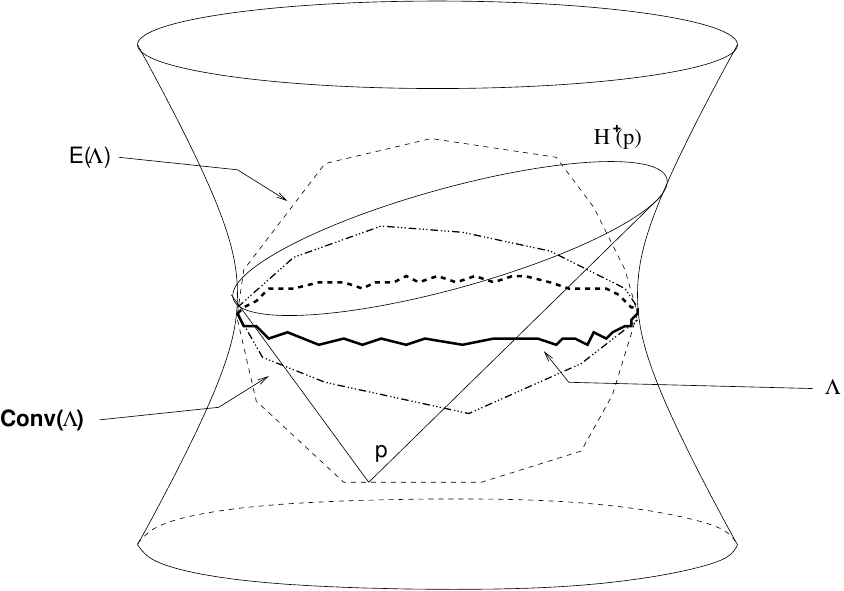}
\centering  \caption{The global situation. The hyperboloid represents the boundary of an affine domain of
$\AdS_{n+1}$ containing the invisible domain. The limit set $\Lambda$ is represented by a topological circle turning around
the hyperboloid, and $\op{Conv}(\Lambda)^\circ$ is a convex subset inside the (dual) convex subset $E(\Lambda)$. The
future-dual plane $H^+(p)$ for $p$ in the past boundary component $\cH^-(\Lambda)$ is a support hyperplane of $S^+(\Lambda)$.}
  \label{fig:convexhull}
\end{figure}

\begin{remark}
For every $p$ in $\cH^-(\Lambda)$, $H^+(p)$ is a support hyperplane for $\op{Conv}(\Lambda)$, but it could be at a point in $\Lambda$.
Elements of $H^-(\Lambda)$ that are support hyperplanes for $\op{Conv}(\Lambda)$ at a point inside $\AdS_{n+1}$, \ie in $\S^+(\Lambda)$
form an interesting subset of $\cH^-(\Lambda)$, the \textit{initial singularity set} (cf. \cite{benbon}).
\end{remark}

\subsection{Global hyperbolicity}

\begin{defi}
\label{def:gh}
A spacetime $(M,g)$ is \textit{globally hyperbolic} (abbreviation GH) if:
\begin{itemize}
\item $(M, g)$ is \textit{causal}, \ie contains no timelike loop,
\item for every $p$, $q$ in $M$, the intersection $J^+(p) \cap J^-(q)$ is empty or compact.
\end{itemize}
\end{defi}

\begin{defi}
Let $(M, g)$ be a spacetime.
A \textit{Cauchy hypersurface} is a closed acausal subset $S \subset M$ that intersects
every inextendible causal curve in $(M, g)$ in one and only one point.

A \textit{Cauchy time function} is a time function $T: M \to \RR$ such that every level set $T^{-1}(a)$ is a Cauchy hypersurface
in $(M, g)$.
\end{defi}

\begin{theo}[\cite{choquet}, \cite{sanchez, sanchez2, sanchezcausal}]
Let $(M, g)$ be a spacetime. The following assertions are equivalent:
\begin{enumerate}
\item $(M, g)$ is globally hyperbolic,
\item $(M, g)$ contains a Cauchy hypersurface,
\item $(M, g)$ admits a Cauchy time function,
\item $(M, g)$ admits a smooth Cauchy time function.
\end{enumerate}
\end{theo}


In a GH spacetime, the Cauchy hypersurfaces are homeomorphic one to the other. In particular, if one of them is
compact, all of them are compact.

\begin{defi}
A spacetime $(M ,g)$ is \textit{globally hyperbolic spatially compact} (abbrev. GHC) if it contains a closed
Cauchy hypersurface.
\end{defi}

\begin{prop}
A spacetime $(M, g)$ is GHC if and only if it contains a time function $T: M \to \RR$ such that
every level set $T^{-1}(a)$ is compact.\fin
\end{prop}

\subsection{Cosmological time functions}
\label{s.cosmological-time}

In any spacetime $(M,g)$, one can define the \textit{cosmological time
  function\/} as follows (see~\cite{cosmic}):

\begin{defi}
The cosmological time function of a spacetime $(M,g)$ is the function
$\tau:M\rightarrow [0,+\infty]$ defined by
$$\tau(x):=\mbox{Sup}\{ L(c) \mid c \in {\mathcal R}^-(x) \},$$  where
${\mathcal R}^-(x)$ is the set of past-oriented  causal curves starting at
$x$, and $L(c)$ is  the Lorentzian length of the causal curve $c$.
\end{defi}


\begin{defi}
\label{d.regular}
A spacetime $(M,g)$ is \textit{CT-regular} with cosmological time
function $\tau$ if
\begin{enumerate}
\item $M$ has \textit{finite existence time,\/}  $\tau(x) < \infty$ for
   every $x$ in $M$,
\item for every past-oriented inextendible causal curve $c: [0, +\infty)
  \rightarrow M$, $\lim_{t \to \infty} \tau(c(t))=0$.
\end{enumerate}
\end{defi}

\begin{teo}[\cite{cosmic}]
\label{teo.cosmogood}
If a spacetime $(M,g)$ has is CT-regular, then
\begin{enumerate}
\item $M$ is globally hyperbolic,
\item $\tau$ is a time function, i.e. $\tau$ is continuous and is strictly
  increasing along future-oriented causal curves,
\item for each $x$ in $M$, there is at least one realizing  geodesic, \ie a future-oriented timelike geodesic $c:
  (0, \tau(x)] \rightarrow M$ realizing the distance from the "initial
  singularity", that is, $c$ has unit speed, is geodesic, and
  satisfies:
$$ c(\tau(x))) = x \mbox{ and } \tau(c(t)) = t \mbox{ for every }t
$$
\item $\tau$ is locally Lipschitz, and admits first and second derivative
almost everywhere.
\end{enumerate}\fin
\end{teo}

However, $\tau$ is not always a Cauchy time function (see the comment after Corollary 2.6 in \cite{cosmic}).

A very nice feature of CT-regularity is that is is preserved by isometries (and thus, by Galois automorphisms):

\begin{prop}[Proposition 4.4 in \cite{merigot}]
\label{pro.cosmolift}
Let  $(\wt{M}, \tilde{g})$ be a CT-regular spacetime. Let $\Gamma$ be a torsion-free discrete group of isometries of
$(\wt{M},\tilde{g})$ preserving the time orientation. Then, the action of $\Gamma$ on  $(\wt{M}, \tilde{g})$
is properly discontinuous. Furthermore, the quotient spacetime $(M,g)$ is CT-regular. More precisely,
if $\mathrm{p}: \wt{M} \to M$ denote the quotient map, the cosmological times $\tilde{\tau}: \wt{M} \to [0, +\infty)$
and $\tau: M \to [0, +\infty)$ satisfy:
\[
\tilde{\tau}=\tau \circ \mathrm{p} \]
\end{prop}

Recall that in this section $\Lambda$ denotes a non-purely lightlike topological achronal $(n-1)$-sphere in $\partial\AdS_{n+1}$.

\begin{prop}[Proposition 11.1 in \cite{ABBZ}]
\label{pro.adsregular}
The GH-regular AdS domain $E(\Lambda)$ is CT-regular.\fin
\end{prop}

Hence, according to Theorem~\ref{teo.cosmogood}, GH-regular domains are globally hyperbolic. Furthermore:

\begin{defi}
The region $\{ \tau < \pi/2 \}$ is denoted $E_0^-(\Lambda)$ and called the {\it past tight} region of
$E(\Lambda)$.
\end{defi}

\begin{prop}[Proposition 11.5 in \cite{ABBZ}]
\label{pro.toutads}
Let $x$ be an element of the past tight region $E_0^-(\Lambda)$. Then, there is an unique realizing geodesic for $x$. More
precisely, there is one and only one element $r(x)$ in the past horizon $\cH^-(\Lambda)$ - called the cosmological retract of $x$ -
such that the segment $(r(x), x]$ is a timelike geodesic whose Lorentzian length is precisely the cosmological time $\tau(x)$.\fin
\end{prop}

\begin{prop}[Proposition 11.6 in \cite{ABBZ}]
\label{pro.tightads}
Let $c: (0, T] \rightarrow E_0^-(\Lambda)$ be a future
oriented timelike geodesic whose initial extremity  $p:=\lim_{t\to 0}c(t)$ is in the past
horizon $\cH^-(\Lambda)$. Then the following assertions are equivalent.
\begin{enumerate}
\item For every $t\in(0,T]$, $c_{\mid[0,t]}$ is a realizing geodesic for the
  point $c(t)$.
\item There exists $t\in (0, T]$ such that $c_{\mid[0,t]}$ is a realizing
  geodesic for the point $c(t)$.
\item $c$ is orthogonal to a support hyperplane of $E(\Lambda)$ at
  $p:=\lim_{t\to 0}c(t)$.
\end{enumerate}
\end{prop}

The following Proposition was known in the case $n=2$
(\cite{mess1, benbon}, and was implicitly admitted in
the few previous papers devoted to the higher dimensional case
(for example, \cite{ABBZ, merigot}):

\begin{prop}
\label{pro:taupi}
The past tight region $E^-_0(\Lambda)$ is the past in $E(\Lambda)$ of the future component $S^+(\Lambda)$ of the convex core
(in particular, it contains $\op{Conv}(\Lambda)^\circ$).
The restriction of the cosmological time to $E^-_0(\Lambda)$ is a Cauchy time,
taking all values in $(0, \pi/2)$.
\end{prop}

\begin{proof}
Let $x$ be an element of $E^-_0(\Lambda)$. According to Propositions~\ref{pro.toutads}, \ref{pro.tightads} there is a realizing geodesic $(r(x), x]$ orthogonal to a spacelike support hyperplane
$H$ tangent to $\cH^-(\Lambda)$ at $r(x)$. As described in Sect.~\ref{sec:moreconvex}, this support hyperplane is the hyperplane $H^-(p)$ past-dual to an element $p$ of $S^+(\Lambda)$.
The realizing geodesic is contained in the geodesic $\theta \mapsto c(\theta) = \cos(\theta)r(x) + \sin(\theta)p(x)$ (cf. Remark~\ref{rk:geodtheta}). For $\theta$ in $(0, \pi/2)$ sufficiently closed to
$\pi/2$, $c(\theta)$ belongs to $\op{Conv}(\Lambda) \subset E(\Lambda)$, and since $E(\Lambda)$ is convex, every $c(\theta) \;\;\; (\theta \in (0, \pi/2)$ lies in $E(\Lambda)$. Moreover,
according to Proposition~\ref{pro.tightads}, for every $\theta_0$ in $(0, \pi/2)$, the restriction of $c$ to $(0, \theta_0)$ is a realizing geodesic. Hence:
$$\forall \theta \in (0, \pi/2), \;\;\; \tau(c(\theta)) = \theta$$
Hence, every value in $(0, \pi/2)$ is attained by $\tau$. Moreover, $x$ lies in the past of $p(x)$, hence of $S^+(\Lambda)$. We have:
$$E^-_0(\Lambda) \subset I^-(S^+(\Lambda)) \cap E(\Lambda)$$
Inversely, for every $p$ in $I^-(S^+(\Lambda)) \cap E(\Lambda)$, there is a (not necessarily unique) realizing geodesic $c: (0, \tau(x)) \to E(\Lambda)$ such that $c(\tau(x)) = x$ (cf. item $(3)$
in Theorem~\ref{teo.cosmogood}). Then, the curve $c$ being a timelike geodesic inextendible (in $E(\Lambda)$) in past, for $t \to 0$ the points $c(t)$ converge to a limit point $c(0)$ in $\cH^-(\Lambda)$.
If $\tau(x) \geq \pi/2$, on the one hand we observe that $c(\pi/2)$ lies in the past of $x = c(\tau(x))$, hence in $I^-(S^+(\Lambda))$. On the other hand:
$$\langle c(\pi/2) \mid c(0) \rangle = 0$$
Therefore, $c(\pi/2)$ is dual to an element of $\cH^-(\Lambda)$ and belongs to $S^+(\Lambda)$. But it is a contraction since $S^+(\Lambda)$ is achronal and $c(\pi/2) \in I^-(S^+(\Lambda))$.
Hence $\tau(x) < \pi/2$, \ie:
$$I^-(S^+(\Lambda)) \cap E(\Lambda) \subset E^-_0(\Lambda)$$

In order to conclude, we have to prove that $\tau$ is a Cauchy time function.
Let $c_0: (a,b) \to E^-_0(\Lambda)$ be an inextendible future oriented causal curve. The image of $\tau \circ c_0$ is an interval $(\alpha, \beta)$. According to item $(2)$ of Definition~\ref{d.regular}, $\alpha=0$.
We aim to prove $\beta = \pi/2$, hence we assume by contradiction that $\beta < \pi/2$.
The curve $c$ is contained in the compact subset $Cl\left(E(\Lambda)\right)$ of $\AdS_{n+1} \cup \partial\AdS_{n+1} \subset \Ein_{n+1}$, hence admits a future limit point $c(b)$ in $\AdS_{n+1} \cup \partial\AdS_{n+1}$.
If $c(b)$ lies in $\Ein_n = \partial\AdS_{n+1}$, then it is in $\Lambda$ (cf. Proposition~\ref{pro.inter-boundary}). Some element of $E(\Lambda)$ (for example, $c(\frac{a+b}2)$) would be causally related to an element of $\Lambda$.
This contradiction shows that $c(b)$ lies in $\AdS_{n+1}$; more precisely, in the boundary of $E^-_0(\Lambda)$ in $\AdS_{n+1}$. Since $c$ is future oriented, it follows that $c(b)$ has to be an element of the future boundary $S^+(\Lambda)$.

For every $t$ in $(a,b)$, we denote by $r(t)$ the cosmological retract $r(c(t))$ of $c(t)$, and we consider the unique realizing geodesic segment $\delta_t := (r(t), c(t))$. We extract a subsequence $t_n$ converging to $b$ such that $r(t_n)$ converges to an element
$r_0$ of $Cl\left(\cH^-(\Lambda)\right) = \cH^-(\Lambda) \cup \Lambda$. Then, $\delta_{t_n}$ converge to a geodesic segment $\delta_0 = (r_0, c(b))$. Since every $\delta_{t_n}$ is timelike, $\delta_0$ is non-spacelike.

For every $t$ in $(a,b)$ we have $c(t) = \cos\tau(c(t))r(t) + \sin\tau(c(t))p(t)$ (where $p(t)$ is the dual of the hyperplane orthogonal to the realizing geodesic at $r(t)$, see above). Hence:
$$\langle r(t_n) \mid c(t_n) \rangle = -\cos\tau(c(t_n))$$
At the limit:
$$\langle r_0 \mid c(b) \rangle = -\cos(\beta) < 0 \;\; \op{ (since } \beta < \pi/2 \op{)}$$

It follows that $\delta_0$ is not lightlike, but timelike. Since timelike geodesics in $\AdS_{n+1}$ remain far away from $\partial\AdS_{n+1}$, it follows that $r_0$ lies in $\cH^-(\Lambda)$.

Finally, every $\delta_{t_n}$ is orthogonal to a support hyperplane at $r(t_n)$, hence at the limit $\delta_0$ is orthogonal to a support hyperplane, which is spacelike since $\delta_0$ is timelike.
According to Proposition~\ref{pro.tightads}, $\delta_0$ is a realizing geodesic. At the beginning of the proof, we have shown that every realizing geodesic can be extended to a timelike geodesic of length $\pi/2$ entirely
contained in $E^-_0(\Lambda)$, hence there is an element $p_0$ in $S^+(\Lambda) \cap H^+(r_0)$ such that the geodesic $(r_0, p_0)$ contains $\delta_0$, in particular $c(b)$. Hence $[c(b), p_0]$ is a non-trivial timelike geodesic
segment joining two elements of the achronal subset $S^+(\Lambda)$, contradiction.

This contradiction proves $\beta = \pi/2$, \ie that the restriction of $\tau$ to every inextendible causal curve is surjective. In other words, $\tau$ is a Cauchy time function. The Proposition is proved.
\end{proof}

\begin{lemma}
\label{l:tauC1}
The restriction of $\tau$ to $E_0^-(\Lambda)$ is $C^{1,1}$ (\ie differentiable with locally Lipschitz derivative), and the realizing geodesics are orthogonal to the level sets of $\tau$.
\end{lemma}

\begin{proof}
Let $x$ be an element of $E_0^-(\Lambda)$, and let $(r(x), x]$ be the unique realizing geodesic for $x$. As proven during the proof of Proposition~\ref{pro:taupi}, there is an element $p(x)$ of $S^+(\Lambda)$ such
that $(r(x), p(x))$ is a timelike geodesic containing $x = \cos(\tau(x))r(x) + \sin(\tau(x))p(x)$ and entirely contained in $E^-_0(\Lambda)$.

Let $U$ be the affine domain $U(p(x))$; the past component $H$ of $U$ is a support hyperplane of $\cH^-(\Lambda)$ at $r(x)$ (see Definition~\ref{def.affine}, Remark~\ref{rk:pastcomponent}).
Let $\tau_0: U \to (0, \pi)$ the cosmological time function of $U$: for every $y$ in $U$,
$\tau_1(y)$ is the Lorentzian distance between $y$ and $H$. Let $W$ be the future of $r(x)$ in $U$,
and let $\tau_0$ be the cosmological time function in $W$: for every $y$ in $W$ $\tau_0(y)$
is the the Lorentzian length of the timelike geodesic $[r(x), y]$. We have:
$$\tau_0(x) = \tau(x) = \tau_1(x)$$
Moreover:
$$\forall y \in W, \;\;\; \tau_0(y) \leq \tau(y) \leq \tau_1(y)$$
A direct computation shows that $\tau_0$ and $\tau_1$ have the same derivative at $x$: by a standard
argument (see for example \cite[Proposition~1.1]{caffarelli}) it follows that $\tau$ is differentiable at $x$, with derivative $d_x\tau = d_x\tau_0 = d_x\tau_1$. Furthermore, the gradient of $\tau_0$ and $\tau_1$ at $x$ is $-\nu(x)$
where $\nu(x)$ is the future-oriented timelike vector tangent at $x$ to the realizing geodesic $[x, r(x))$ of Lorentzian norm $-1$, \ie:
$$\forall v \in T_xW, \;\;\; -\langle v \mid \nu(x) \rangle = d_x\tau_0(v)x = d_x\tau(v)$$

Therefore, $-\nu(x)$ is also the Lorentzian gradient of $\tau$. It follows that realizing geodesics are orthogonal to the level sets of $\tau$.

In order to prove that $\tau$ is $C^{1,1}$, \ie that $\nu$ is locally Lipschitz, we adapt the argument used in the flat case in \cite{barflat}. We consider first the restriction of $\nu$ to the level set $S_{\pi/4} = \tau^{-1}(\pi/4)$ equipped with the induced Riemannian metric. For every $x$ in $S_{\pi/4}$ we have $x = \frac{r(x)+p(x)}{\sqrt{2}}$.
Observe that $\frac{p(x)-r(x)}{\sqrt{2}}$ is then an element of $\RR^{2,n}$
of norm $-1$, orthogonal to $x$, hence representing an element of $T_x\AdS_{n+1}$.
This tangent vector is future-oriented and orthogonal to $S_{\pi/4}$: hence
$\frac{p(x)-r(x)}{\sqrt{2}}$ represents $\nu(x)$.

Let $c: (-1, 1) \to S_{\pi/4}$ be a $C^1$ curve in $S_{\pi/4}$. Since $r$ is the projection on $\cH^-(\Lambda)$, and since $\cH^-(\Lambda)$ is locally Lipschitz, the path $r\circ c$ is differentiable almost everywhere in $(-1,1)$. We denote by $\dot{r}$, $\dot{p}$, $\dot{\nu}$ the derivatives of $r$, $p$, $\nu = \frac{p-r}{\sqrt{2}}$ along $c$.
Almost everywhere, we have:
\begin{eqnarray*}
  \mathrm{q}_{2,n}(\dot{\nu}) &=& \mathrm{q}_{2,n}(\frac{\dot{p}-\dot{r}}{\sqrt{2}}) \\
   &=& \frac12(\mathrm{q}_{2,n}(\dot{p}) + \mathrm{q}_{2,n}(\dot{r}) -2\langle \dot{r} \mid \dot{p}\rangle)
\end{eqnarray*}

But the derivative of $c$ is:
\begin{eqnarray*}
  \mathrm{q}_{2,n}(\dot{c}) &=& \mathrm{q}_{2,n}(\frac{\dot{r}+\dot{p}}{\sqrt{2}}) \\
   &=& \frac12(\mathrm{q}_{2,n}(\dot{r}) + \mathrm{q}_{2,n}(\dot{p}) +2\langle \dot{r} \mid \dot{p}\rangle)
\end{eqnarray*}

Now, since $\cH^-(\Lambda)$ is locally convex, the quantity $\langle \dot{r} \mid \dot{p}\rangle$, wherever it is defined, is nonnegative. Therefore:
$$\mathrm{q}_{2,n}(\dot{\nu}) \leq \mathrm{q}_{2,n}(\dot{c})$$

It follows that $\nu$ is $1$-Lipschitz along $S_{\pi/4}$.

On other level sets $S_t=\tau^{-1}(t)$ with $t \in (0, \pi/2)$, every element is of the form
$x = \cos(t)r(x) + \sin(t)p(x)$, and $x_{\pi/4} = \frac{r(x)+p(x)}{\sqrt{2}}$
is a point in $S_{\pi/4}$. Geometrically, $x_{\pi/4}$ is the unique point in the realizing geodesic for $x$ at cosmological time $\pi/4$. The unit normal vectors $\nu(x)$ and $\nu(x_{\pi/4})$ are parallel one to the other along the realizing geodesic $(r(x), p(x))$, hence, the variation of $\nu(x)$ along $S_t$ is controlled by the distortion of the map $x \to x_{\pi/4}$ and the variation
of $\nu$ along $S_{\pi/4}$. The lemma follows.
\end{proof}



\subsection{GH-regular and quasi-Fuchsian representations}

Let $\Gamma$ be a fini\-tely generated torsion-free group, and let $\rho: \Gamma \to \SO_0(2,n)$ be a faithful, discrete representation, such that
$\rho(\Gamma)$ preserves $\Lambda$. According to Proposition~\ref{pro.cosmolift},
the quotient space $M_\rho(\Lambda):=\rho(\Gamma)\backslash{E}(\Lambda)$ is globally hyperbolic. Observe that moreover
Cauchy hypersurfaces of $M_\rho(\Lambda)$ are quotients of Cauchy hypersurfaces in
$E(\Lambda)$, which are contractible (since graphs of maps from $\DD^n$ into $\RR$). Hence,
if $\Gamma$ has cohomological dimension $\geq n$, the Cauchy hypersurfaces are compact, \ie $M_\rho(\Lambda)$ is spatially compact - and
the cohomological dimension of $\Gamma$ is eventually precisely $n$.

Inversely, in his celebrated preprint~\cite{mess1, mess2}, G. Mess\footnote{Mess only deals with the case where $n=2$, but his arguments also
apply in higher dimension. For a detailed proof see~\cite[Corollary
11.2]{barbtz1}} proved that any
globally hyperbolic spatially compact AdS spacetime embeds isometrically in such a quotient space $\Gamma\backslash{E}(\Lambda)$.

\begin{defi}
\label{def.gh}
Let $\Gamma$ be a torsion-free discrete group. A representation $\rho: \Gamma \rightarrow \SO_0(2,n)$
is \textit{GH-regular} if it is faithfull, discrete and preserves a non-empty GH-regular domain $E(\Lambda)$ in $\partial\AdS_{n+1}$.
If moreover the $(n-1)$-sphere $\Lambda$ is acausal, then the representation is \textit{strictly GH.}
\end{defi}

\begin{defi}
\label{def.ghc}
A (strictly) GH-regular representation $\rho: \Gamma \to \SO_0(2,n)$ is  (strictly) \textit{GHC-regular}
if the quotient space $\rho(\Gamma)\backslash{E}(\Lambda)$ is spatially compact.
\end{defi}

Hence a reformulation of Mess result is:
\begin{prop}
\label{pro.qf=ghc}
A representation $\rho: \Gamma \to \SO_{0}(2,n)$ is GHC-regular if and only if it is the holonomy of
a GHC AdS spacetime.\fin
\end{prop}

There is an interesting special case of strictly GHC-regular representations: the case of \textit{quasi-Fuchsian representations.}

\begin{defi}
\label{def.quasifuchs}
A strictly GHC-regular representation $\rho: \Gamma \to \SO_0(2,n)$ is \textit{quasi-Fuchsian} if $\Gamma$ is isomorphic to
a uniform lattice in $\SO_0(1,n)$.
\end{defi}

This terminology is motivated by the analogy with the hyperbolic case.

There is a particular case: the case where  $\Lambda$ is a ``round sphere'' in $\partial\AdS_{n+1}$, \ie the boundary of a totally
geodesic spacelike hypersurface $\SS(v^\perp) \cap \AdS_{n+1}$:

\begin{defi}
\label{def.fuchs}
A \textit{Fuchsian} representation  $\rho: \Gamma \to \SO_0(2,n)$ is the composition of the natural inclusions $\Gamma \subset \SO_0(1,n)$
and $\SO_0(1,n) \subset \SO_0(2,n)$, where in the latter $\SO_0(1,n)$ is considered as the stabilizer in $\SO_0(2,n)$ of a point in $\AdS_{n+1}$.
\end{defi}

In other words, a quasi-Fuchsian representation is Fuchsian if and only if it admits a global fixed point in $\AdS_{n+1}$.

\subsection{The space of timelike geodesics}
\label{sub.timelikegeodesic}
Timelike geodesics in $\AdS_{n+1}$ are intersections between $\AdS_{n+1} \subset \RR^{2,n}$ and $2$-planes $P$ in $\RR^{2,n}$
such that the restriction of $\mathrm{q}_{2,n}$ to $P$ is negative definite. The action of $\SO_0(2,n)$ on negative $2$-planes is transitive, and the stabilizer of the $(u, v)$-plane is $\SO(2) \times \SO(n)$. Therefore, the space of timelike geodesics is the symmetric space:
$$\cT_{2n} := \SO_0(2,n)/\SO(2) \times \SO(n)$$

$\cT_{2n}$ has dimension $2n$. We equip it by the Riemannian metric $g_\cT$ induced by the Killing form of $\SO_0(2,n)$. It is well known
that $\cT_{2n}$ has nonpositive curvature, and rank $2$: the maximal flats (\ie totally geodesic embedded Euclidean subspaces) have dimension $2$. It is also naturally hermitian. More precisely: let $\mathcal G = \mathfrak{so}(2,n)$ be the Lie  algebra of $G=\SO_0(2,n)$, and let $\cK$ be the Lie algebra of the maximal compact subgroup $K:=\SO(2) \times \SO(n)$. We have the Cartan decomposition:
$$\mathcal G = \cK \oplus \cK^\perp$$
where $\cK^\perp$ is the orthogonal of $\cK$ for the Killing form. Then, $\cK^\perp$ is naturally identified with the tangent space at the origin of $G/K$. The adjoint action of the $\SO(2)$ term in the stabilizer defines a $K$-invariant complex structure on $\cK^\perp \approx T_K(G/K)$ that propagates through left translations to an integrable complex structure $J$ on $\cT_{2n} = G/K$. Therefore, $\cT_{2n}$
is naturally equipped with a structure of $n$-dimensional complex manifold, together with a $J$-invariant Riemannian metric, \ie an hermitian structure.

Let us consider once more the achronal $(n-1)$-dimensional topological sphere $\Lambda$.
Then, it is easy to prove that every timelike geodesic in $\AdS_{n+1}$ intersects $E(\Lambda)$ (cf. Lemma 3.5 in \cite{merigot}), and since $E(\Lambda)$ is convex, this intersection is connected, \ie is a single inextendible timelike geodesic of $E(\Lambda)$. In other words,
one can consider $\cT_{2n}$ as the space of timelike geodesics of $E(\Lambda)$.

Let $\rho: \Gamma \to \SO_0(2,n)$ be a GH-regular representation preserving $\Lambda$. The (isometric) action of
$\rho(\Gamma)$ on $\cT_{2n}$ is free and proper, and the quotient $\cT_{2n}(\rho):=\rho(\Gamma)\backslash\cT_{2n}$ is naturally identified with the
space of inextendible timelike geodesics of $M_\rho(\Lambda)=\rho(\Gamma)\backslash E(\Lambda)$.

\begin{defi}
\label{def:gaussmap}
Let $S$ be a differentiable Cauchy hypersurface in a GH-regular spacetime $M_\rho(\Lambda)$ of dimension $n+1$. The \textit{Gauss map} of $S$
is the map $\nu: S \to \cT_{2n}(\rho)$ that maps every element $x$ of $M_\rho(\Lambda)$ to the unique timelike geodesic of $M_\rho(\Lambda)$ orthogonal to $S$ at $x$.

When $S$ is $C^{1,1}$ (for example, a level set $\tau^{-1}(t)$ of the cosmological time for $t<\pi/2$), then
one can define for every $C^1$ curve $c$ in $S$ the \textit{Gauss length} as the length in $\cT_{2n}(\rho)$ of the Lipschitz curve
$\nu \circ c$. It defines on $S$ a length metric, called the \textit{Gauss metric} (of course, if
$S$ is $C^r$ with $r\geq2$, then $\nu$ is $C^{r-1}$, and the Gauss metric is a $C^{r-1}$ Riemannian metric).
\end{defi}

Since every timelike geodesic intersects $S$ at most once, the Gauss map is always injective. The image of the Gauss map
is actually the set of timelike geodesics that are orthogonal to $S$. Since every timelike geodesic intersects $S$, it follows easily that
the image of the Gauss map is closed, and that the Gauss map is actually an embedding.

\begin{remark}
\label{r:defgausstau}
For every $t<\pi/2$, let $\Sigma_t(\tau)$ be the image by the Gauss map of the cosmological level set $\tau^{-1}(t)$.
According to Lemma~\ref{l:tauC1}, $\Sigma_t(\tau)$ is the space of realizing geodesics. In particular, it does not depend on $t$.
We will denote by $\Sigma(\tau)$ this closed embedded submanifold, and call it the \textit{space of cosmological geodesics.}
\end{remark}

\subsection{Split $\AdS$ spacetimes}
\label{s:split}
Let $(p,q)$ be a pair of positive integers such that $p+q=n$.
Let $(x_0, x_1, \ldots , x_p, y_0, y_1, \ldots , y_q)$ be a coordinate system such that the quadratic form is:
$$ - x_0^2 + x_1^2 + ... + x_p^2 - y_0^2 + y_1^2 + \ldots + y_q^2$$

Let $G_{p,q} \approx \SO_{0}(1,p) \times \SO_{0}(1,q)$ be the subgroup of $\SO_{0}(2,n)$ preserving the splitting
$\RR^{2,n}=\RR^{1,p} \oplus \RR^{1,q}$ where $\RR^{1,q}$ is the subspace $\{x_0 = x_1 = \ldots = x_p = 0 \}$ and  $\RR^{1,p}$
the subspace $\{ y_0 = y_1 = \ldots = y_q = 0\}$.

Let $\Lambda_p$ (respectively $\Lambda_q$) be the subset $\SS(\cC^+_p)$ (respectively $\SS(\cC^+_q)$) of (the Klein model of) $\Ein_n$ where:
$$\cC^+_p := \{   - x_0^2 + x_1^2 + ... + x_p^2 = 0, x_0 > 0, y_0 = y_1 = \ldots = y_q = 0 \}$$
and
$$\cC^+_q := \{  - y_0^2 + y_1^2 + \ldots + y_q^2 = 0, y_0 > 0, x_0 = x_1 = \ldots = x_p = 0 \}$$

Observe that $\Lambda_p$, $\Lambda_q$ are topological spheres of dimension respectively $p-1$, $q-1$. Moreover, for every pair of elements $\op{x}$, $\op{y}$ in
$\Lambda_p \cup \Lambda_q$ the scalar product $\langle \op{x} \mid \op{y} \rangle$ is nonpositive. Hence, according to Corollary~\ref{cor.list},
$\Lambda_p \cup \Lambda_q$ is achronal. Moreover, every point in $\Lambda_p$ is linked to every point in $\Lambda_q$ by a unique lightlike geodesic segment contained in $\Ein_n$.

\begin{lemma}
\label{l:EestConv}
The invisible domain $E(\Lambda_p \cup \Lambda_q)$ is the interior of the convex hull of $\Lambda_p \cup \Lambda_q$.
\end{lemma}

\begin{proof}
Clearly:
$$\op{Conv}(\cC^+_p) =  \{   - x_0^2 + x_1^2 + ... + x_p^2 \leq 0, x_0> 0, y_0 = y_1 = \ldots = y_q = 0 \} $$
Similarly:
$$\op{Conv}(\cC^+_q) =  \{    - y_0^2 + y_1^2 + \ldots + y_q^2 \leq 0, y_0 > 0, x_0 = x_1 = \ldots = x_p = 0  \} $$

Therefore, $\op{Conv}(\Lambda_p \cup \Lambda_q)$ is the projection by $\SS$ of the set of points $(x_0, x_1, \ldots , x_p, y_0, y_1, \ldots , y_q)$ satisfying the following inequalities:
\begin{eqnarray*}
 - x_0^2 + x_1^2 + ... + x_p^2 & \leq & 0 \\
- y_0^2 + y_1^2 + \ldots + y_q^2 & \leq & 0 \\
 x_0 & \geq & 0 \\
y_0 & \geq & 0
\end{eqnarray*}

According to Remark~\ref{rk.nice} $\op{Conv}(\Lambda_p \cup \Lambda_q)^\circ$ is contained in $E(\Lambda_p \cup \Lambda_q)$.
Inversely, let $\op{z} = (x_0, x_1, \ldots , x_p, y_0, y_1, \ldots , y_q)$ be an element of $\RR^{2,n}$ representing an element $z_0$ of $E(\Lambda_p \cup \Lambda_q)$. Then, by definition of
$E(\Lambda_p \cup \Lambda_q)$, the scalar product $\langle \op{z} \mid \op{x} \rangle$ is negative for every element $\op{x}$ of $\cC^+_p$. It follows
that $(x_0, x_1, \ldots , x_p)$ must lie in the future cone of $\RR^{1,p}$, \ie :
\begin{eqnarray*}
 - x_0^2 + x_1^2 + ... + x_p^2 & < & 0 \\
 x_0 &> & 0
\end{eqnarray*}

Similarly, since $\langle \op{z} \mid \op{x} \rangle < 0$ for every element $\op{x}$ of $\cC^+_q$:
\begin{eqnarray*}
- y_0^2 + y_1^2 + \ldots + y_q^2 & < & 0 \\
y_0 & > & 0
\end{eqnarray*}

The lemma follows.
\end{proof}

Let $\Lambda_{p,q}$ be the intersection in $\SS(\RR^{2,n})$ between $\op{Conv}(\Lambda_p \cup \Lambda_q)$ and $\SS(\cC_n) \approx \Ein_n$. Let $(x_0, x_1, \ldots , x_p, y_0, y_1, \ldots , y_q)$ be an element of $\RR^{2,n}$ representing an element of $\Lambda_{p,q}$.
According to the proof of Lemma~\ref{l:EestConv} we must have $ - x_0^2 + x_1^2 + ... + x_p^2 \leq 0$ and $- y_0^2 + y_1^2 + \ldots + y_q^2 \leq 0$, and since  $(x_0, x_1, \ldots , x_p, y_0, y_1, \ldots , y_q)$ lies in $\cC_n$, these quantities must vanish.
Hence, the inequalities defining $\Lambda_{p, q}$ are:
\begin{eqnarray*}
 - x_0^2 + x_1^2 + ... + x_p^2 & = & 0 \\
- y_0^2 + y_1^2 + \ldots + y_q^2 & = & 0 \\
 x_0 & \geq & 0 \\
y_0 & \geq & 0
\end{eqnarray*}

Therefore, $\Lambda_{p,q}$ is the union of $\Lambda_p$, $\Lambda_q$, and every lightlike segment joining a point of $\Lambda_p$ to a point of $\Lambda_q$: it is achronal, but not acausal! Topologically, $\Lambda_{p,q}$ is the join of two spheres, therefore, its a sphere of dimension $1 + (p-1) + (q-1) = n-1$. It is not an easy task to figure out how it fits inside $\Ein_n =\partial\AdS_{n+1}$.

For that purpose, we consider the coordinates $(r, \theta, a_1, \ldots , a_p, b_1, \ldots , b_p)$ on $\AdS_{n+1}$ such that $x_0 = r\cos\theta$, $y_0=r\sin\theta$, $x_i=ra_i$, $y_i = rb_i$. According to Proposition~\ref{p.causal-structure}, the $n+1$-uple
$(a_1, \ldots , a_p, b_1, \ldots, b_q, 1/r)$ describes the upper hemisphere  $\DD^n = \{ a^2_1 + \ldots + a_p^2 + b_1^2 + \ldots + b_q^2 + 1/r^2 = 1 \}$ in the Euclidean sphere of $\RR^{n+1}$ of radius $1$, and $\AdS_{n+1}$ is conformally
isometric to the product $\SS^1 \times \DD^n$ with the metric $-\op{d\theta}^{2} + \op{ds}^{2}$, where $\op{ds}^2$ is the round
metric on $\DD^n $.

In these coordinates, the inequalities defining $E(\Lambda_p \cup \Lambda_q)$ established in the proof of Lemma~\ref{l:EestConv} become:
\begin{eqnarray}
 & 0 < \theta < \pi/2 \\
& a_1^2 + \ldots + a_p^2  <  \cos^2\theta &\\
& b_1^2 + \ldots + b^2_q < \sin^2\theta &
\end{eqnarray}

Let $\DD_0^q$ be the subdisk of $\DD^n$ defined by $a_1 = \ldots = a_p =0$, and let $\DD^p_0$ be the subdisk defined by $b_1 = \ldots = b_q = 0$. For every $\op{x}$ in $\DD^n$, let $d_p(\op{x})$ be the distance of $\op{x}$ to $\DD^q_0$, and define similarly the "distance to $\DD^p_0$" function
$d_q: \DD^n \to [0, +\infty)$. Observe that since $\DD^p_0$ and $\DD^q_0$ both contain the North pole $(0, \ldots, 0, 1)$ of $\DD^n$, and since every point in $\DD^n$ is at distance at most $\pi/2$ of the North pole, $d_p$ and $d_q$ takes value in $[0, \pi/2)$.
Now, observe that the following identities hold:
\begin{eqnarray}
a_1^2 + \ldots + a_p^2  & = &  \sin^2d_p(a_1, \ldots , a_p, b_1, \ldots, b_q, 1/r) \\
b_1^2 + \ldots + b_q^2  & = &  \sin^2d_q(a_1, \ldots , a_p, b_1, \ldots, b_q, 1/r)
\end{eqnarray}

It follows that $E(\Lambda_p \cup \Lambda_q)$ is the domain in $\AdS_{n+1} \approx \SS^1 \times \DD^n$ comprising points
$(\theta, \op{x})$ such that:
$$ d_q(\op{x}) < \theta < \pi/2 - d_p(\op{x})$$
In the terminology of Definition~\ref{r.f-f+}, it means that the lifting $\widetilde{E}(\widetilde{\Lambda_p \cup \Lambda_q})$ is defined by the functions $f^- = d_q$ and $f^+ = \pi/2 - d_p$. These functions extend uniquely as $1$-Lipschitz maps $f^\pm: \overline{\DD}^n \to [0, \pi/2]$.

The boundary $\partial\DD^n = \SS^{n-1}$ is totally geodesic in $\overline{\DD}^n$, and $\partial\DD^q_0$, $\partial\DD^p_0$ are totally geodesic spheres of dimensions $p$, $q$, respectively.
Let $\delta_p: \partial\DD^n \to [0, \pi/2]$ (respectively $\delta_q: \partial\DD^n \to [0, \pi/2]$) be the function "distance to $\partial\DD^q_0$" (respectively "distance to $\partial\DD^p_0$").
It follows from equations\footnote{These equations naturally extend to the boundary $\partial\DD^n$.} $(4)$, $(5)$ that every point of $\partial\DD^q_0$ is at distance $\pi/2$ of $\partial\DD^p_0$.
Hence:
$$\delta_p + \delta_q = \pi/2$$

In other words, the restrictions of $f^-$ and $f^-$ to $\partial\DD^n$ coincide and are equal to $\delta_q = \pi/2 - \delta_p$.
The restriction of $f^-=f^+$ to $\partial\DD^p_0$ vanishes, and the graph of this resctriction is $\Lambda_p$. The restriction
of $f^-=f^+$ to $\partial\DD^q_0$ is the constant map of value $\pi/2$, and the graph is $\Lambda_q$.
The graph of $f^\pm: \partial\DD^n \to \SS^1$ is $\Lambda_{p,q}$, which is therefore an achronal sphere in $\Ein_n$.

Clearly, $\Lambda_{p,q}$ is preserved by $G_{p,q}$.
Let $\Gamma$ be a cocompact lattice of $G_{p,q} \approx \SO_{0}(1,p) \times \SO_{0}(1,q)$.
The inclusion $\Gamma \subset G_{p,q} \subset \SO_{0}(2,n)$  is a GH-regular representation, but non-strictly since the invariant achronal limit set
$\Lambda_{p,q}$ is not acausal.
According to Proposition~\ref{pro.cosmolift}, the quotient space $M_{p,q}(\Gamma) := \Gamma\backslash E(\Lambda_{p,q})$ is a GH spacetime. Actually, the Cauchy surfaces of $M_{p,q}(\Gamma)$ are quotients by $\Gamma$ of the graph of a $1$-Lipschitz map $f: \DD^n \to \SS^1$, hence they are $K(\Gamma, 1)$ (since $\DD^n$ is contractible). On the other hand, the quotient of $\HH^p \times \HH^q$ is a $K(\Gamma, 1)$ too. Since $\Gamma$ is a cocompact lattice, it follows that every $K(\Gamma, 1)$ - in particular, the Cauchy hypersurfaces in $M_{p,q}(\Gamma)$ - are compact. The inclusion $\Gamma \subset \SO_0(2,n)$ is therefore GHC-regular.

\begin{defi}
\label{def:split}
The quotient space $M_{p,q}(\Gamma)$ is a \textit{split $\AdS$ spacetime.} The representation $\rho: \Gamma \to \SO_0(2,n)$
is a \textit{split GHC-regular representation of type $(p,q)$.}
\end{defi}

\begin{remark}
The split $\AdS$ spacetimes of dimension $2+1$ are precisely the \textit{Torus universes} studied in \cite{carlip}. Observe indeed that the lattice in $\SO_0(1,1) \times \SO_0(1,1) \approx \RR^2$ is isomorphic to $\ZZ^2$, and the Cauchy surfaces are indeed tori.
\end{remark}

\subsection{Crowns}
A particular case of split $\AdS$ spacetime is the case $p=q=1$ (and therefore, $n=2$). Then, the topological spheres $\Lambda_p$ and $\Lambda_q$ have dimension $0$, \ie, are pair of points $\Lambda_p = \{ \op{x}^-, \op{y}^- \}$ and $\Lambda_q = \{ \op{x}^+, \op{y}^+\}$. The topological circle $\Lambda_{p,q}$ is then piecewise linear; more precisely, it is the union of the four lightlike segments $[\op{x}^-, \op{x}^+]$, $[\op{x}^+, \op{y}^-]$, $[\op{y}^-, \op{y}^+]$, $[\op{y}^+, \op{x}^-]$.
The invisible domain $E(\Lambda_{p,q})$ is then an ideal tetrahedron, interior of the  convex hull of the four ideal points $\{\op{x}^-, \op{y}^-, \op{x}^+, \op{y}^+\}$. This tetrahedron has six edges; four of them as the lightlike segments forming $\Lambda_{p,q}$, and the two others are the spacelike geodesics $(\op{x}^-, \op{y}^-)$ and $(\op{x}^+, \op{y}^+)$ of $\AdS_{n+1}$  (see Figure~\ref{fig:crown}). Observe that  $[\op{x}^-, \op{x}^+]$ and $[\op{y}^-, \op{y}^+]$ are future oriented, whereas $[\op{x}^+, \op{y}^-]$ and $[\op{y}^+, \op{x}^-]$ are past oriented.

 \begin{figure}
  \centering
  \includegraphics{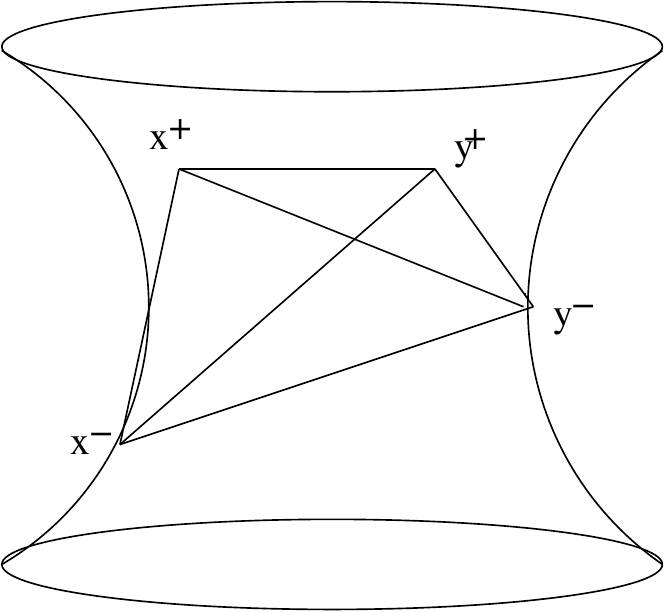}
\centering  \caption{Picture of a crown. The hyperboloid represent the boundary of an affine domain of
$\AdS_{n+1}$ containing the realm of the crown.}
  \label{fig:crown}
\end{figure}

More generally:

\begin{defi}
For every integer $n\geq2$, a \textit{crown} of $\Ein_n$ is $4$-uple $\mathfrak{C} = (\op{x}^-, \op{y}^-, \op{x}^+, \op{y}^+)$ in $\Ein_n$
such that:
\begin{itemize}
\item $\langle \op{x}^- \mid \op{x}^+\rangle = \langle \op{x}^- \mid \op{y}^+\rangle = 0$
\item $\langle \op{y}^- \mid \op{x}^+\rangle = \langle \op{y}^- \mid \op{y}^+\rangle = 0$
\item $\langle \op{x}^- \mid \op{y}^-\rangle < 0$
\item $\langle \op{x}^+ \mid \op{y}^+\rangle < 0$
\item the lightlike segment $[\op{x}^-, \op{x}^+]$ is future oriented.
\end{itemize}
The subset $\{\op{x}^-, \op{y}^-, \op{x}^+, \op{y}^+\}$ is then an achronal subset of $\Ein_n$. The invisible domain
$E(\{\op{x}^-, \op{y}^-, \op{x}^+, \op{y}^+\})$ is called the \textit{realm of the crown}, and denoted by $E(\sC)$.
The convex hull of $\{\op{x}^-, \op{y}^-, \op{x}^+, \op{y}^+\}$ is denoted by $\op{Conv}(\sC)$.
\end{defi}

\begin{remark}
Let $\sC =  (\op{x}^-, \op{y}^-, \op{x}^+, \op{y}^+)$ be a crown in $\Ein_n$, and let $x^-$, $y^-$, $x^+$, $y^+$ be elements
of $\RR^{2,n}$ representing the vertices of the crown. Let $V(\sC)$ be the linear space spanned by $x^+$, $x^-$, $y^-$, $y^+$.
The restriction of $\mathrm{q}_{2,n}$ to $V(\sC)$ has signature $(2,2)$, and $\SS(V(\sC))$ is the unique totally geodesic copy of $\Ein_2$ in $\Ein_n$ containing $\sC$.
\end{remark}

\begin{remark}
\label{rk:realsplit}
Let $Z$ be the stabilizer in $\SO_0(2,n)$ of a crown. It preserves the orthogonal sum $V(\sC) \oplus V(\sC)^\perp$. It is isomorphic to
the product $A \times \SO(n-2)$, where $A$ is a maximal $\RR$-split abelian subgroup of $\SO_0(2,2)$, hence of $\SO_0(2,n)$.
Therefore, $Z$ is the centralizer of $A$, and it has finite index in the normalizer $N$ of $A$. It follows that
 the space of crowns is naturally a finite covering over the space $G/N$ of maximal flats in the symmetric space $\cT_{2n}=\SO_0(2,n)/SO(2)\times SO(n)$.
\end{remark}

\begin{remark}
Up to an isometry, one can assume that the crown $\sC = (\op{x}^-, \op{y}^-, \op{x}^+, \op{y}^+) $ is represented by:

\begin{eqnarray*}
x^+ & = & (1, 0, 1, 0, 0, \ldots , 0)\\
y^+ & = & (1, 0, -1, 0, 0, \ldots , 0)\\
x^-  & = & (0, 1, 0, 1, 0, \ldots , 0)\\
y^-  & = & (0, 1, 0, -1, 0, \ldots , 0)
\end{eqnarray*}

According to Proposition~\ref{pro.proj}, the realm $E(\sC)$ is defined by the inequalities:
\begin{eqnarray*}
& x_1 - u < 0& \\
& -x_1-u < 0 & \\
& x_2 - v <0 & \\
& - x_2 - v < 0 & \\
&-u^2 - v^2 + x_1^2 + \ldots + x_n^2<0&
\end{eqnarray*}

Hence, by:
$$\vert x_1 \vert < u, \;\;\; \vert x_2 \vert < v, \;\;\;  -u^2 - v^2 + x_1^2 + \ldots + x_n^2<0$$

Observe that the last inequality is implied by the two previous inequations when $n=2$.
\end{remark}

If $n=2$, the realm of a crown $\sC$ coincide with the interior of $\op{Conv}(\sC)$ (Lemma~\ref{l:EestConv}), but this is obviously not true for $n>2$ since $\op{Conv}(\sC)$ is always $3$-dimensional.



\section{Acausality of limit sets of Gromov hyperbolic groups}
In all this section, $\Gamma$ is a torsion-free Gromov hyperbolic group, and
$\rho: \Gamma \to \SO_0(2,n)$ a GHC-regular representation, with limit set $\Lambda$.
By hypothesis, $E(\Lambda)$ is not empty, therefore $\Lambda$ is not purely lightlike.

\subsection{Non-existence of crowns}

\begin{prop}
\label{pro:nocrown}
The limit set $\Lambda$ contains no crown.
\end{prop}

\begin{proof}
Let $\sC = (\op{x}^-, \op{y}^-, \op{x}^+, \op{y}^+)$ be a crown contained in $\Lambda$.
Let $F(\sC)$ be the subset of $\cT_{2n}$ comprising timelike geodesics containing a segment $[\op{p}^-, \op{p}^+]$ with $\op{p}^\pm \in (\op{x}^\pm, \op{y}^\pm)$.
Let $A$ be the  maximal $\RR$-split abelian subgroup stabilizing $\sC$, \ie the subgroup of the stabilizer $Z$ of $\sC$ acting trivially on $V(\sC)^\perp$ (cf. Remark \ref{rk:realsplit}).
Then, $F(\sC)$ is an orbit of the action of $A$ in $\cT_{2n}$. Therefore, $F(\sC)$ is a flat in the symmetric space $\cT_{2n}$.

Let $\Sigma(\tau)$ be the space of cosmological geodesics in $E^-_0(\Lambda)$
(cf. Remark~\ref{r:defgausstau}).

\textit{Claim: $\Sigma(\tau)$ contains $F(\sC)$.}

Let $\op{p}^+$, $\op{p}^-$ be elements of $(\op{x}^+, \op{y}^+)$, $(\op{x}^-, \op{y}^-)$. The closure of $E(\Lambda)$ contains $\op{Conv}(\Lambda)$, in particular, it contains $\op{p}^\pm$. On the other hand, $\langle \op{x}^+ \mid \op{p}^- \rangle = 0$, hence $\op{p}^+$ does not lie in $E(\Lambda)$. Therefore, $\op{p}^-$ is an element of $\cH^-(\Lambda)$.

Observe that $\langle \op{p}^- \mid \op{p}^+ \rangle = 0$. Hence, $\op{p}^-$ lies in the hyperplane $H^-(\op{p}^+)$ past-dual to $\op{p}^+$. Now, since $\op{p}^+$ lies in $\op{Conv}(\Lambda)$, we have $\langle \op{p}^+ \mid \op{y} \rangle \leq 0$ for every $\op{y}$ in  $E(\Lambda)$. Therefore,  $H^-(\op{p}^+)$ is a support hyperplane of $\cH^-(\Lambda)$ at $\op{p}^-$, orthogonal to the timelike geodesic $[\op{p}^-, \op{p}^+]$. According to Proposition~\ref{pro.tightads}, $(\op{p}^-, \op{p}^+)$ is a realizing geodesic, hence an element of $\Sigma(\tau)$. The claim follows.

Consider now the Gauss metric on $\Sigma(\tau)$ (cf. Definition~\ref{def:gaussmap}).
According to the claim, $\Sigma(\tau)$ contains the Euclidean plane $F(\sC)$. Since $F(\sC)$ is totally geodesic in $\cT_{2n}$, it is also totally geodesic in $\Sigma(\tau)$.

On the other hand, the group $\Gamma$ acts on $\Sigma(\tau)$, and the quotient of this action
is compact, since this quotient is the image by the Gauss map of compact surface in $M_\rho(\Lambda)$. Hence, $\Sigma(\tau)$ is quasi-isometric to $\Gamma$, and therefore, Gromov hyperbolic.
It is a contradiction since a Gromov hyperbolic metric space cannot contain a $2$-dimensional flat.
\end{proof}

\subsection{Compactness of the convex core}

In Sect.~\ref{sub.convexhull}, we have seen that, up to a lifting in $\wt\AdS_{n+1}$, the convex core $\op{Conv}(\Lambda)$ (respectively the invisible domain $E(\Lambda)$) can be defined as the region between the graphs of functions $F^\pm: \DD^n \to \RR$ (respectively $f^\pm: \DD^n \to \RR$) such that (cf. Proposition~\ref{pro.F-F+}):
\begin{equation}
f^- \leq F^- \leq F^+ \leq f^+
\end{equation}
where the inequality $F^- \leq F^+$ is strict as soon as $\rho: \Gamma \to \SO_0(2,n)$ is not Fuchsian.

\begin{prop}
\label{pro:convexcore}
The left and right inequalities in $(6)$ are strict, \ie for every $\op{x}$ in $\DD^n$, we have:
$$f^-(\op{x}) < F^-(\op{x}) \leq F^+(\op{x}) < f^+(\op{x})$$
\end{prop}

\begin{proof}
Assume by contradiction that $f^\pm(\op{x}) = F^\pm(\op{x})$ for some $\op{x}$ in $\DD^n$.
It means that some element $x$ of $\op{Conv}(\Lambda) \cap \AdS_{n+1}$ is on the boundary of $E(\Lambda)$.
This element is a linear combination $x = t_1x_1 + \ldots + t_kx_k$ where $k \geq 2$, $t_i$ are positive real numbers and $x_i$ elements of $\cC_n \subset \RR^{2,n}$ such that the projections $\SS(x_i)$ belong to $\Lambda$. Moreover, since $x$ lies in $\AdS_n$, we have $\langle x_a \mid x_b \rangle < 0$ for some integers $a$, $b$. Since $x$ lies in the boundary of $E(\Lambda)$, there is an element $x_0$ of $\Lambda$ such that:
\begin{eqnarray*}
0 & = & \langle x_0 \mid x \rangle\\
  & = & t_1\langle x_0 \mid x_1 \rangle + \ldots + t_k\langle x_0 \mid x_k \rangle
\end{eqnarray*}
Since $\Lambda$ is achronal, each $\langle x_0 \mid x_i \rangle$ is nonpositive, therefore vanishes.
In particular:
\begin{itemize}
\item $\langle x_0 \mid x_a\rangle = \langle x_0 \mid x_b\rangle = 0$
\item $\langle x_a \mid x_b\rangle < 0$
\end{itemize}

Reverting the time orientation if necessary, we can assume that $x$ lies in the past horizon $\cH^-(\Lambda)$. Moreover, we can assume without loss of generality that $x$ is actually equal
to $x_a+x_b$, after rescaling if necessary $x_a$, $x_b$ so that $x_a + x_b$ has norm $-1$, \ie
lies in $\AdS_{n+1}$.

Consider now any element $y_0$ of $E^-_0(\Lambda)$ in the future of $x$, \ie such that $(x, y_0)$ is
a future oriented timelike segment. More precisely, we can select $y_0$ such that the timelike segment $[x, y_0]$ is orthogonal to the segment $[x_a, x_b]$. Let $t_0$ be the cosmological time at $y_0$, let $S_0$ be the cosmological level set $\tau^{-1}(t_0)$, and let $d_0$ the induced metric on $S_0$: this metric is complete since $S_0$ admits a compact quotient.

Let $P$ be the $3$-subspace of $\RR^{2,n}$ spanned by $y_0$, $x$ and $x_0$: by construction,
$P$ is orthogonal to $x_a$ and $x_b$.
Then, $A := \SS(P) \cap \ADS_{n+1}$ is a totally geodesic copy of $\AdS_2$.
The restriction of $\tau$ to $A \cap E^-_0(\Lambda)$ is still a Cauchy time function, and $S_0 \cap A$ is a spacelike path which contains $\SS(y_0)$. Moreover, there is a sequence $y_n$ in $S_0 \cap A$ converging to $\SS(x_0)$.

Let $K_0 \subset S_0$ be a compact fundamental domain for the action of $\rho(\Gamma)$ on $S_0$.
There is a sequence $g_n = \rho(\gamma_n)$ in $\rho(\Gamma)$ such $z_n = g_n y_n$ converge to $\bar{z}$ in $K_0$. We define:
\begin{eqnarray*}
  a_n &=& g_n x_a \\
  b_n &=& g_n x_b \\
  q_n &=& g_n x_0\\
  x_n &=& g_n x = a_n + b_n
\end{eqnarray*}

Up to a subsequence, we can assume that $\SS(a_n)$, $\SS(b_n)$, $\SS(q_n)$ converge to elements $\bar{a}$, $\bar{b}$, $\bar{q}$ of $\Lambda$, and that $\SS(x_n)$ converge to an element $\bar{x}$ of the segment $[\bar{a}, \bar{b}]$. At this level, it could happen that this segment is reduced to one point, \ie $\bar{a}=\bar{b}$; but we will prove that it is not the case.

\textit{Claim:  $\bar{x}$ lies in $\AdS_{n+1}$.}

Indeed, since every $x_n$ belongs to $\cH^-(\Lambda)$, if the limit $\bar{x}$ does not lie in $\AdS_{n+1}$, then it is an element of $\Lambda$. The segment $[\bar{x}, \bar{z}]$, limit of the timelike segments $[x_n, z_n]$, would be causal, and  the element $\bar{z}$ of $K_0 \subset E(\Lambda)$ would be causally related to the
element $\bar{x}$ of $\Lambda$: contradiction.

Therefore, $\bar{x}$ lies in $\cH^-(\Lambda)$. It follows in particular that $\bar{a} \neq \bar{b}$.
Consider now the iterates $p_n := g_ny_0$ of $y_0$. They belong to $S_0$. Up to a subsequence, we can assume that the sequence $(p_n)_{n \in \NN}$ admits a limit $\bar{p}$. Since $d_0$ is complete and the $y_n$ converge to a point in $\partial\AdS_{n+1}$, the distance $d_0(y_n, y_0)$ converge to $+\infty$. Therefore, $d_0(z_n, p_n) = d_0(g_n y_n, g_n y_0) = d_0(y_n, y_0)$ is unbounded: the limit $\bar{p}$
is at infinity, \ie an element of $\Lambda$.

The four points $\bar{q}$, $\bar{a}$, $\bar{b}$, $\bar{p}$ in $\Ein_n$ satisfy:
\begin{itemize}
  \item $\langle \bar{q} \mid \bar{a} \rangle = \langle \bar{q} \mid \bar{b} \rangle = 0$ (since
  $\langle x_0 \mid x_a \rangle = \langle x_0 \mid x_b \rangle = 0$),
  \item $\langle \bar{a} \mid \bar{b} \rangle < 0$ (since $(\bar{a}, \bar{b})$ contains the element $\bar{x}$ of $\AdS_{n+1}$
  \item $\langle \bar{p} \mid \bar{a} \rangle = \langle \bar{p} \mid \bar{b} \rangle = 0$ (since
  every $p_n$ lies in $a_n^\orth \cap b_n^\orth$).
\end{itemize}

Now observe that in every iterate $A_n = g_nA_0$, the timelike geodesic $\Delta_0$ containing $[x_n, z_n]$ disconnects $A_n$, and that the ideal points $q_n$, $p_n$ lie on (the boundary of) different components of $A \setminus \Delta_0$.

It follows that $\bar{p} \neq \bar{q}$. Observe that $\bar{q}$, $\bar{p}$ lies in the isotropic cone of $\bar{a}^\orth \cap \bar{b}^\orth$, which has signature $(1, n-1)$. Moreover, every $p_n$, $q_n$ lies in the future of $x_n$: it follows that $\bar{q}$, $\bar{p}$ lies in the same connected component of
the isotropic cone of $\bar{a}^\orth \cap \bar{b}^\orth$ (with the origin removed); therefore:
$$\langle \bar{p} \mid \bar{q} \rangle < 0$$
It follows that $(\bar{a}, \bar{b}, \bar{p}, \bar{q})$ is a crown. It contradicts Proposition~\ref{pro:nocrown}.
\end{proof}

\subsection{Proof of Theorem~\ref{thm:hyperbolicanosov}}
In this section, we prove Theorem~\ref{thm:hyperbolicanosov}:

\begin{theo}
\label{thm:hyperbolicanosov2}
Let $\rho: \Gamma \to \SO_0(2,n)$ be a GHC-regular representation, where $\Gamma$ is a Gromov hyperbolic group. Then the achronal limit set $\Lambda$ is acausal, \ie $\rho$ is $(\SO_0(2,n), \Ein_n)$-Anosov.
\end{theo}

\begin{proof}
We equip the convex domain $E(\Lambda)$ with its \textit{Hilbert metric:} for every element $x$, $y$
in $E(\Lambda) \subset \ADS_{n+1}$, the hilbert distance $d_h(x,y)$ is defined to be the cross-ratio
$[a; x; y; b]$ where $a$, $b$ are the intersections between $\partial E(\Lambda)$ and the projective line in $\SS(\RR^{2, n})$ containing $x$ and $y$. The Hilbert metric is of course $\rho(\Gamma)$-invariant.

Assume by contradiction that $\Lambda$ is \textit{not} acausal. Then, it contains a lightlike segment $[x, y]$ with $x \neq y$. We can assume wlog that this segment is maximal, \ie that $[x, y]$ is precisely the intersection between $\Lambda$ and a projective line in $\Ein_n \subset \SS(\RR^{2,n})$.
Let $P$ be a projective subplane of $\SS(\RR^{2,n})$ containing $[x, y]$ and an element $z$ of
$\op{Conv}(\Lambda)^\circ$. The intersection $P \cap \op{Conv}(\Lambda)^\circ$ is a convex domain containing the ideal triangle
$x$, $y$, $z$, with a side $[x, y]$ contained at infinity. Let $u$ be an element in the segment $(x,y)$. For every $t>0$, let $x_t$ (respectively $y_t$) be the element of the segment $[z, x)$ (respectively $[z, y)$) such that $d_h(z, x_t) = t$ (respectively $d_h(z, y_t) = t$), and let $u_t$ be the intersection $[z,u] \cap [x_t, y_t]$.
Observe that $[z, x_t] \cup [x_t, y_t] \cup [y_t, z]$ is a geodesic triangle for $d_h$.
Now, an elementary computation shows (see the proof of Proposition 2.5 in \cite{benoist1}):
$$ \lim_{t \to +\infty} d_h(u_t, [z, x_t] \cup [z, y_t]) = +\infty$$
It implies that $\op{Conv}(\Lambda) \setminus \Lambda$, equipped with the restriction of
$d_h$, is not Gromov hyperbolic.

But, on the other hand, the quotient of $\op{Conv}(\Lambda) \setminus \Lambda$ by $\rho(\Gamma)$ is compact. Indeed, according to Proposition~\ref{pro:convexcore}, the future boundary $S^+(\Lambda)$ and the past boundary $S^-(\Lambda)$ of the convex core are contained in $E(\Lambda)$. Their projections in $M_\rho(\Lambda)$ are therefore compact achronal hypersurfaces, bounding a compact region $C$, which is precisely the quotient of $\op{Conv}(\Lambda) \setminus \Lambda$.

Since $\Gamma$ is Gromov hyperbolic, $(\op{Conv}(\Lambda) \setminus \Lambda, d_h)$ should be Gromov hyperbolic. Contradiction.
\end{proof}

\section{Limits of Anosov representations}
This section is entirely devoted to the proof of the Theorem~\ref{thm:main2}, that we restate here for
the reader's convenience:\\

\textbf{Theorem 1.2.}
\textit{Let $n \geq 2$, and let $\Gamma$ be a Gromov hyperbolic group of cohomological dimension $\geq n$.
Then, the modular space $\op{Rep}_0(\Gamma, \SO_0(2,n))$ of $(\SO_0(2,n), \Ein_n)$-Anosov representations is open and closed in the modular space $\op{Rep}(\Gamma, \SO_0(2,n))$.}
\\

We recall that one important step of the proof will be be to show that under these hypothesis, $\Gamma$ is the fundamental group
of a closed manifold, and that its cohomological dimension is eventually $n$ (cf. Remark \ref{rk:cohom}).

Let $\Gamma$ be as in the hypothesis of the Theorem a Gromov hyperbolic group of cohomological dimension $\geq n$.
The fact that $\op{Rep}_0(\Gamma, \SO_0(2,n))$ is open in $\op{Rep}_0(\Gamma, \SO_0(2,n))$
is well-known (cf. Theorem 1.2 in \cite{guichard3}), hence our task is to prove that it is
a closed subset.

Let $\rho_k: \Gamma \to \op{SO}_0(2,n)$ be a sequence of $(\SO_0(2,n), \Ein_n)$-Anosov representations converging to a representation $\rho_\infty: \Gamma \to \SO_0(2,n)$.

\begin{prop}
\label{pro:disfaith}
The limit representation $\rho_\infty: \Gamma \to \SO_0(2,n)$ is discrete and faithfull.
\end{prop}

\begin{proof}
Since $\Gamma$ is Gromov hyperbolic and non-elementary, it contains no nilpotent normal subgroup (see \cite{ghysharpe}). Hence, by a classical argument, the limit $\rho_\infty: \Gamma \to \SO_0(2,n)$ is discrete and faithfull (cf. Lemma~1.1 in \cite{goldmil}).

Actually, we give a sketch of the argument, since we will need later a slightly more elaborate version
of this argument. The key point is that $\SO_0(2,n)$, as any Lie group, contains a neighborhood $W_0$ of the identity such that every discrete subgroup generated by elements in $W_0$ is contained in
a nilpotent Lie subgroup of $\SO_0(2,n)$. In particular, such a discrete subgroup is
nilpotent, and there is an uniform bound $N$ for the residue class (\ie the length of the lower central series) of these nilpotent groups.

Assume that $\op{Ker}(\rho_\infty(\Gamma))$ is non-trivial. Then it is a normal subgroup. For any finite subset $F$ of $\op{Ker}(\rho_\infty(\Gamma))$, there is an integer $k_0$ such that $k\geq k_0$
implies that $\rho_k(F)$ is contained in $W_0$, hence nilpotent of residue class $\leq N$. It follows that $\op{Ker}(\rho_\infty(\Gamma))$ is nilpotent, contradiction: the representation $\rho_\infty$ is
faithful.

Let $\bar{G}_\infty$ be the closure of $\rho_\infty(\Gamma)$, and let $\bar{G}^0_\infty$ be the identity component of $\bar{G}_\infty$: it is a normal subgroup of $\bar{G}_\infty$, and it is generated by any neighborhood of the identity. Therefore, $\rho_\infty(\Gamma) \cap W_0$
generates a dense subgroup of $\bar{G}^0_\infty$. On the other hand, any expression of the form:
\begin{equation}
\label{eq:infty}
[\rho_\infty(\gamma_1), [\rho_\infty(\gamma_2), [... [\rho_\infty(\gamma_{N}), \rho_\infty(\gamma_{N+1})] ...]]]
\end{equation}
is the limit for $k \to +\infty$ of:
\begin{equation}
\label{eq:n}
[\rho_k(\gamma_1), [\rho_k(\gamma_2), [... [\rho_k(\gamma_{N}), \rho_k(\gamma_{N+1})] ...]]]
\end{equation}
For $k$ sufficiently big, every $\rho_k(\gamma_i)$ belongs to $W_0$ and $\rho_k(\Gamma)$ is discrete, hence (\ref{eq:n}) is trivial. The limit (\ref{eq:infty}) is trivial too. It follows that
$\bar{G}^0_\infty$ is nilpotent. Then, $\rho_\infty^{-1}(\rho_\infty(\Gamma) \cap \bar{G}^0_\infty)$
is a nilpotent normal subgroup of $\Gamma$. It is a contradiction, unless $\bar{G}^0_\infty$ is trivial, \ie unless $\rho_\infty(\Gamma)$ is discrete.
\end{proof}

An immediate consequence of the representations $\rho_k$ being Anosov is the existence of a $\rho_k(\Gamma)$-equivariant map $\xi: \partial_\infty\Gamma \to \Ein_n$ whose
image is a closed $\rho_k(\Gamma)$-invariant acausal subset $\Lambda_k$ (cf. \cite{guichard3}). According to the Remark \ref{rk:extension}, for every integer $k$,
there is a $\rho_k(\Gamma)$-invariant achronal topological $(n-1)$-sphere ${\Lambda}^+_k$, which is not pure lightlike since it contains the acausal subset $\Lambda_k$.
Therefore, every $\rho_k$ is a GH-regular representation. The Cauchy hypersurfaces of the associated GH spacetimes are contractible (since the universal coverings are topological disks
embedded in regular domains of $\AdS_{n+1}$) and have fundamental groups isomorphic to $\Gamma$. Since $\Gamma$ has cohomological dimension $\geq n$, these Cauchy hypersurfaces are compact: the $\rho_k$ are GHC-regular representations.

The $\rho_k(\Gamma)$-invariant spheres $\hat{\Lambda}_k$ are graphs of locally $1$-Lipschitz maps $f_k: \SS^{n-1} \to \SS^1$.
It follows easily by Ascoli-Arzela Theorem that, up to a subsequence, $\rho_\infty(\Gamma)$ preserves
the graph a of locally $1$-Lipschitz map $f_\infty: \SS^{n-1} \to \SS^1$, ie. an achronal sphere $\Lambda_\infty$.

\begin{lemma}
\label{le:notpure}
$\Lambda_\infty$ is not purely lightlike.
\end{lemma}

\begin{proof}
Assume not. Then, $\Lambda_\infty$ is the union of lightlike geodesics joining two
antipodal points $\op{x}_0$ and $-\op{x}_0$ in $\Ein_n$. Let $G_0$ be the stabilizer
in $\SO_0(2,n)$ of $\pm \op{x}_0$: the image $\rho_\infty(\Gamma)$ is a discrete subgroup
of $G_0$.

According to Remark~\ref{rk:minkconf}, the group $G_0$ is isomorphic to the group of conformal transformations of the Minkowski space $\op{Mink}(\op{x}_0) \approx \RR^{1,n-1}$. There is an exact sequence:
$$1 \to \RR^{1,n-1} \to G_0 \to \RR \times \SO_0(1,n-1) \to 1$$
where the left term is the subgroup of translations of $\RR^{1,n-1}$ and the right term
the group of conformal linear transformations of $\RR^{1,n-1}$.
Let $L: G_0 \to \RR \times \SO_0(1,n-1)$ be the projection morphism. Let $\bar{L}$ be the closure
in $\RR \times \SO_0(1,n-1)$ of $L(\rho_\infty(\Gamma))$, and let $\bar{L}_0$ be the identity component of
$\bar{L}$. Considering as in the proof of Proposition~\ref{pro:disfaith} an open domain $V_0$
in $G_0$ such that any discrete group generated by elements of $V_0$ is nilpotent, and using as a trick the fact that conjugacies in $G_0$ by homotheties in $\RR^{1,n-1}$
can reduce at an arbitrary small scale translations in $\RR^{1,n-1}$, one proves that $\Gamma \cap L^{-1}(L(\rho_\infty(\Gamma)) \cap\bar{L}_0)$ is a normal nilpotent subgroup of $\Gamma$
(cf. Theorem 1.2.1 in \cite{cardalbo}). Therefore, it is trivial: $L(\rho_\infty(\Gamma))$ is a discrete subgroup of $\RR \times \SO_0(1,n-1)$.

Now we consider $\RR \times \SO_0(1,n-1)$ as the group of isometries of the Riemannian product
$\RR \times \HH^{n-1}$. By what we have just proved, the action of $\rho_\infty(\Gamma)$ on $\RR \times \HH^{n-1}$ is properly
discontinuous. On the other hand, $\Gamma$ acts properly and cocompactly on a topological disk of dimension $n$ (a Cauchy hypersurface in $E(\Lambda_k)$ for any $k$),
hence its action on $\RR \times \HH^{n-1}$ is cocompact. It is a contradiction since $\RR \times \HH^{n-1}$ is not Gromov hyperbolic (it contains flats of dimension $2$).
\end{proof}

\textit{Proof of Theorem \ref{thm:main2}. } According to Proposition~\ref{pro:disfaith} and Lemma~\ref{le:notpure}, $\rho_\infty: \Gamma \to \SO_0(2,n)$ is a GH-regular representation. It is actually a GHC-regular representation since Cauchy surfaces in $\rho_\infty(\Gamma)\backslash E(\Lambda_\infty)$ are $K(\Gamma, 1)$ and thus compact since Cauchy surfaces in every $\rho_k(\Gamma)\backslash E(\Lambda_k)$ are compact.
According to Theorem~\ref{thm:hyperbolicanosov2}, the representation $\rho_\infty: \Gamma \to \SO_0(2,n)$ is
$(\SO_0(2,n), \Ein_n)$-Anosov.\fin

\section{Bounded cohomology}

This section is devoted to the proof of:

\textbf{Theorem 1.4.}
\textit{Let $\rho: \Gamma \to \SO_0(2,n)$ be a faithful and discrete representation, where $\Gamma$ is the fundamental group of a negatively curved closed manifold $M$. The following assertions are equivalent:
\begin{enumerate}
\item $\rho$ is $(\SO_0(2,n), \Ein_n)$-Anosov,
\item the bounded Euler class $\eu_b(\rho)$ vanishes.
\end{enumerate}}

For a friendly introduction to bounded cohomology, close to our present concern, see \cite[Section 6]{ghysgroup}.

\subsection{The bounded Euler class}
We have the following central exact sequence:
\begin{eqnarray}
\label{eq:exactsequence}
1 \to \ZZ \to \widetilde{\SO}_0(2,n) \to \SO_0(2,n) \to 1
\end{eqnarray}
where $\ZZ$ is the group\footnote{Observe that $\ZZ$ is not always the center of $\widetilde{\SO}_0(2,n)$, since $-\op{Id}$ is an element of $\SO_0(2,n)$ when $n$ is even.} of deck transformations of the covering $\hat{p}: \uEin_n \to \Ein_n$,
generated by the transformation $\delta$ (cf section~\ref{sub.einuniv}).
Fix the element $x_0 = (0, \op{x}_0)$ in $\uEin_n \approx \RR \times \SS^{n-1}$.
In these coordinates, $\delta$ is the transformation $(\theta, \op{x}) \mapsto (\theta+2\pi, \op{x})$.
Hence, we can define a section $\sigma: \SO_0(2,n) \to \uSO_0(2,n)$, called \textit{canonical section,} which maps every element $g$ of $\SO_0(2,n)$ to the unique element $\sigma(g)$ of $\widetilde{\SO}_0(2,n)$ above $g$ and such that $\sigma(g(x_0))$ lies in the domain:
$$\cW_0 := \{ (\theta, \op{x}) \in \RR \times \SS^{n-1} \;/\; -\pi \leq \theta < \pi \}$$
Observe that $\cW_0$ is a fundamental domain for the action of $\langle \delta \rangle = \ZZ$ on
$\uEin_n$.

For any pair $(g_1, g_2)$ of elements of $\SO_0(2,n)$, we define $c(g_1, g_2)$ as the unique integer
$k$ such that $\sigma(g_1g_2) = \delta^k\sigma(g_1)\sigma(g_2)$.

\begin{lemma}[Compare with Lemma 6.3 in \cite{ghysgroup}]
The $2$-cocyle $c$ takes only the values $-1$, $0$ or $1$.
\end{lemma}

\begin{proof}
Let $x_1 = (\theta_1, \op{x}_1)$ and $x_2 = (\theta_2, \op{x}_2)$ be the images of $x_0$ by $\sigma(g_1)$, $\sigma(g_2)$, respectively. Let $x_3 = (\theta_3, \op{x}_3)$ be the image of
$x_2$ by $\sigma(g_1)$.

-- $(1)$ \textit{If $|\theta_2| \leq d(\op{x}_2, \op{x}_0)$.} It means that $x_2$ is not in $I^\pm(x_0)$. Then, $x_3 = \sigma(g_1)(x_2)$ is not in $I^\pm(x_1)$. Therefore:
$$\mid \theta_3 - \theta_1 \mid \leq d(\op{x}_3, \op{x}_1) \leq \pi$$
implying $\mid \theta_3 \mid \leq 2\pi$. It follows that if $x_3 = \sigma(g_1)\sigma(g_2)(x_0)$ is not already in $\cW_0$, $\delta^\epsilon(x_3)$ for $\epsilon =\pm 1$ does. Hence $c(g_1,g_2) = \epsilon$ is $0$, $-1$ or $1$ as required.

-- $(2)$ \textit{If $\theta_2 > d(\op{x}_2, \op{x}_0)$.} Then, $0 < \pi - \theta_2 < \pi - d(\op{x}_2, \op{x}_0) = d(\op{x}_2, -\op{x}_0)$ where $-\op{x}_0$ is the antipodal point in $\SS^{n-1}$ at distance $\pi$ from $\op{x}_2$. The point $x_2$ is not in $J^\pm((\pi, -\op{x_0}))$, hence its image $x_3$ by $\sigma(g_1)$ is not in $J^\pm((\pi+\theta_1, -\op{x}_1)$. It follows:
$$\mid \theta_3 - (\pi + \theta_1) \mid < d(\op{x}_3, -\op{x}_1) \leq \pi$$
Therefore:
$$\mid \theta_3 \mid < 3\pi$$
Hence, for some $\epsilon = 0$ or $\pm1$ we have that $\delta^\epsilon(x_3)$ lies in $\cW_0$,
and $c(g_1,g_2) = \epsilon$ is $0$, $-1$ or $1$.

-- $(3)$ \textit{If $-\pi \leq \theta_2 < -d(\op{x}_2, \op{x}_0)$.} We apply the same argument that in case $(2)$, by observing that $\op{x}_2$ is then non causally related to $(-\pi, -\op{x}_0)$. Details are left to the reader.
\end{proof}

\begin{defi}
$c$ is a bounded $2$-cocycle. It represents an element of the bounded cohomology space $\op{H}^2_b(\SO_0(2,n), \ZZ)$ called the bounded Euler class.

For any representation $\rho: \Gamma \to \SO_0(2,n)$, the pull-back $\rho^*([c])$ is an
element of $\op{H}^2_b(\Gamma, \ZZ)$, denoted by $\eu_b(\rho)$.
\end{defi}

Of course, $c$ also represents an element of the "classical" cohomological space $\op{H}^2(\SO_0(2,n), \ZZ)$. The associated $2$-cocycle $\op{eu}(\rho)$ represents the obstruction to lift $\rho$ to a representation $\tilde{\rho}: \Gamma \to \widetilde{\SO}_0(2,n)$. Indeed, $\eu(\rho)=0$ means that there is a $1$-cochain $a: \Gamma \to \ZZ$ such that for every $\gamma_1$, $\gamma_2$ in $\Gamma$ we have:
$$c(\rho(\gamma_1), \rho(\gamma_2)) = a(\gamma_1\gamma_2) - a(\gamma_1) - a(\gamma_2)$$
Then, the map $\gamma \to \delta^{a(\gamma)}\sigma(\rho(\gamma))$ is a morphism,
\ie a representation $\tilde{\rho}: \Gamma \to \widetilde{\SO}_0(2,n)$ which is a lift of $\rho$.

Now $\eu_b(\rho)=0$ means that $\eu(\rho)=0$, but also that one can select the $1$-cochain $a$ so that it is \textit{bounded}.
The following proposition is a natural generalization of the fact a group of orientation-preserving homeomorphisms of the circle has a vanishing bounded Euler class if and only if it has a global fixed point (see the end of section 6.3 in \cite{ghysgroup}):

\begin{prop}
\label{pro:euler0}
The bounded Euler class $\eu_b(\rho)$ vanishes if and only if $\rho$ lifts to a representation $\tilde{\rho}: \Gamma \to \widetilde{\SO}_0(2,n)$ such that $\tilde{\rho}(\Gamma)$ preserves a closed
$(n-1)$-dimensional achronal topological sphere in $\uEin_n$.
\end{prop}

\begin{proof}

\textit{Invariant achronal sphere $\Rightarrow$ $\eu_b(\rho)=0$.}

Assume that $\rho$ lifts to a representation $\tilde{\rho}: \Gamma \to \widetilde{\SO}_0(2,n)$
(\ie that $\eu(\rho) = 0$) and that $\tilde{\rho}(\Gamma)$ preserves a closed
$(n-1)$-dimensional achronal topological sphere $\Lambda$ in $\uEin_n$, \ie the graph of a $1$-Lipschitz map $f: \SS^{n-1} \to \RR$. Let $a: \Gamma \to \ZZ$
the map associating to $\gamma$ the unique integer $k$ such that:
$$\tilde{\rho}(\gamma) = \delta^k\sigma(\rho(\gamma))$$
$a$ is the $1$-cochain whose coboundary represents the Euler class of $\rho$, the point is to prove
that $a$ is bounded.

The invariant achronal sphere $\Lambda$ is contained in the closure of an affine domain of $\uEin_n$ (cf. Lemma~\ref{le.achrinj}), \ie in a domain of the form $\{ \theta_0 - \pi \leq \theta \leq \theta_0 + \pi \}$. More precisely, either it is contained in a domain $\delta^q\cW_0$ for some integer $q$, or it contains a point $(q\pi, \op{x})$, in which case $\Lambda$ is contained in the domain $\{ (q-1)\pi \leq \theta < (q+1)\pi \}$. In both cases, there is an integer $q$ such that $\Lambda$ is contained in the union $\cZ_q := \delta^{q-1}\cW_0 \cup \delta^q\cW_0$.

For every $\gamma$ in $\Gamma$, the image of $x_0 = (0, \op{x}_0)$ by $\sigma(\rho(\gamma))$ is a point $(\theta, \op{y}_0)$ with $|\theta| \leq \pi$, hence the intersection between $\cW_0$ and
$\sigma(\rho(\gamma))(\cW_0)$ is non-trivial. Since $\delta$ commutes with $\sigma(\rho(\gamma))$,
the intersection $\cW_q \cap \sigma(\rho(\gamma))(\cW_q)$ is non-empty. \textit{A fortiori,} the same is true for the intersection $\cZ_q \cap \sigma(\rho(\gamma))(\cZ_q)$. However, since $\delta$ acts by adding $2\pi$ on the coordinate $\theta$, the intersection $\cZ_q \cap \delta^r\sigma(\rho(\gamma))(\cZ_q)$ is empty as soon as $r$ is an integer of absolute value $>2$.

On the other hand, we know that $\cZ_q \cap \tilde{\rho}(\gamma)\cZ_q$ is non-empty since $\cZ_q$ contains the invariant sphere $\Lambda$. It follows that the integer $a(\gamma)$ has absolute value at most $2$.

\textit{$\eu_b(\rho)=0$ $\Rightarrow$ Invariant achronal sphere}

Assume now that $eu_b(\rho)$ vanishes, \ie that there is a bounded map $a: \Gamma \to \ZZ$ such that $\gamma \to \delta^{a(\gamma)}\sigma(\rho(\gamma))$ is a representation $\tilde{\rho}: \Gamma \to \widetilde{\SO}_0(2,n)$. Let $\alpha$ be an upper bound for $|a(\gamma)| \; (\gamma \in \Gamma)$. Let $f_{id}: \SS^n \to \RR$ be the null map, and for every element $\gamma$ of $\Gamma$, let $f_\gamma: \SS^n \to \RR$ be the $1$-Lipschitz map whose graph is the image by $\tilde{\rho}(\gamma)$ of the graph of $f_0$.
The graph of $f_\gamma$ contains $\delta^{a(\gamma)}\sigma(\rho(\gamma))(0, \op{x}_0)$, hence a point
of $\theta$-coordinate of absolute value bounded from above by $|a(\Gamma)|+\pi$. Since every $f_\gamma$ is $1$-Lipschitz and since the sphere has diameter $\pi$, there is an uniform upper bound for all the $f_\gamma$. For every $\op{x}$ in $\SS^n$ define:
$$f_\infty(\op{x}) := \op{Sup}_{\gamma \in \Gamma} f_\gamma(\op{x})$$
Then $f_\infty$ is a $1$-Lipschitz map, whose graph is clearly $\rho(\Gamma)$-invariant.
\end{proof}

\subsection{Proof of Theorem~\ref{thm:euler}}
Let $\rho: \Gamma \to \SO_0(2,n)$ be a faithful and discrete representation, where $\Gamma$ is the fundamental group of a negatively curved closed manifold $M$.

According to the Proposition~\ref{pro:euler0}, the bounded Euler class $\eu_b(\rho)$ vanishes if and only if $\rho$ lifts to a representation $\tilde{\rho}: \Gamma \to \widetilde{\SO}_0(2,n)$ such that $\tilde{\rho}(\Gamma)$ preserves a closed $(n-1)$-dimensional achronal topological sphere in $\uEin_n$. According to Theorem~\ref{thm:hyperbolicanosov}, such a sphere, if it exists, must be acausal. The equivalence between items $(1)$ and $(2)$ follows.

\subsection{The case $n=2$}
In this last section, we explain in which way one can deduce from Proposition~\ref{pro:euler0} the following classical result:

\begin{prop}
\label{pro:euPSL}
Let $\rho_1$, $\rho_2$ be two representations of $\Gamma$ into $\op{PSL}(2,\RR)$ such that $\op{eu}_b(\rho_1) = \op{eu}_b(\rho_2).$
Then, $\rho_1$ and $\rho_2$ are semi-conjugated, \ie there is a monotone map $f: \RR\PP^1 \to \RR\PP^1$ such that:
$$\forall \gamma \in \Gamma, \rho_1(\gamma) \circ f = f \circ \rho_2(\gamma)$$
\end{prop}

Let us first recall the definition of the bounded Euler class
for a representation $\rho: \Gamma \to \op{PSL}(2,\RR)$: it is completely similar
to definition we have presented above.

Let $\op{p}: \widetilde{\op{SL}}(2, \RR) \to \op{PSL}(2,\RR)$ be the universal covering.
It acts naturally on the universal covering $\wt{\RR\PP}^1$ of the projective line $\RR\PP^1,$
so that the kernel of $\op{p}$ is the center of $\widetilde{\op{SL}}(2, \RR)$ and
also the Galois group of $\wt{\RR\PP}^1$. We fix a total order $<$ on $\wt{\RR\PP}^1 \approx \RR$
and a generator $\tau$ of $\ker\op{p}$ so that $\tau(x)>x$ for every $x$ in $\wt{\RR\PP}^1$.
Once fixed an element $x_0$ of $\wt{\RR\PP}^1$,
there is still a canonical
section $\sigma: \op{PSL}(2,\RR) \to \widetilde{\op{SL}}(2, \RR)$, which is not a homomorphism,
which associates to any element $g$ of ${\RR\PP}^1$ the unique element $\tilde{g}$ such that:
$$x_0 \leq \tilde{g}x_0 < \tau(x_0)$$

Then, the Euler class of the representation $\rho$ is the bounded cohomology class
represented by the cocycle $c$ defined by:
$$\sigma(\rho(\gamma_1\gamma_2) = \tau^{c(\gamma_1, \gamma_2)}\sigma(\rho(\gamma_1))\sigma(\rho(\gamma_2))$$

Now let $\rho_1$, $\rho_2$ be two representations of $\Gamma$ into $\op{PSL}(2,\RR)$ satisfying the statement of
Proposition \ref{pro:euPSL}: they have the same bounded cohomology class, meaning that, if $c_1$, $c_2$
are the two cocyles defined as above representing the bounded Euler class, we have:
\begin{equation}\label{eq:a}
    c_2(\gamma_1, \gamma_2) = c_1(\gamma_1, \gamma_2) + a(\gamma_1\gamma_2) - a(\gamma_1) - a(\gamma_2)
\end{equation}
where $a: \Gamma \to \ZZ$ is some bounded map.

It has the following consequence: consider the map
$\Gamma \times \widetilde{\op{SL}}(2, \RR) \to \widetilde{\op{SL}}(2, \RR)$
which associates to $(\gamma, \tilde{g})$ the element: 
$$\gamma*\tilde{g} := \tau^{-a(\gamma)}\sigma(\rho_2(\gamma))\tilde{g}\sigma(\rho_1(\gamma))^{-1}$$ Then:
\begin{eqnarray*}
  (\gamma_1\gamma_2)*\tilde{g} &=& \tau^{-a(\gamma_1\gamma_2)}\sigma(\rho_2(\gamma_1\gamma_2))\tilde{g}[\sigma(\rho_1(\gamma_1\gamma_2))]^{-1} \\
    &=& \tau^{-a(\gamma_1\gamma_2) + c_2(\gamma_1, \gamma_2)}\rho_2(\gamma_1)\rho_2(\gamma_2)\tilde{g}[\tau^{c_1(\gamma_1, \gamma_2)}\sigma(\rho_1(\gamma_1))\sigma(\rho_1(\gamma_2))]^{-1}\\
    &=& \tau^{- a(\gamma_1) - a(\gamma_2)}\tilde{g}[\sigma(\rho_1(\gamma_1))\sigma(\rho_1(\gamma_2))]^{-1}\mbox{(see \eqref{eq:a})}\\
    &=&\gamma_1*(\gamma_2*\tilde{g})
\end{eqnarray*}

Now the key point is that $\widetilde{\op{SL}}(2, \RR)$ is a model for the universal
anti-de Sitter space $\wt\AdS_3$. Indeed, $-\op{det}$ defines
on the space $\op{Mat}(2, \RR)$ of two-by-two matrices a quadratic form of signature $(2,2)$, which is preserved by the following action of $\op{SL}(2, \RR) \times \op{SL}(2, \RR)$:
$$\forall g_1, g_2 \in \op{SL}(2, \RR), \forall A \in \op{Mat}(2,\RR), (g_1, g_2).A := g_1Ag_2^{-1}$$
The kernel of this action is the group $I$ of order two generated by $(-\op{Id}, -\op{Id})$,
where $\op{Id}$ denote the identity matrix.
Hence there is a natural isomorphism between $\SO_0(2,2)$ and $\op{SL}(2, \RR) \times \op{SL}(2, \RR)/I.$

Therefore, the action $*$ we have defined is an isometric action of $\Gamma$ on $\wt\AdS_3$,
hence induces a representation $\tilde{\rho}: \Gamma \to \wt\SO_0(2,2)$
Furthermore, the fact that the map $a$ involved in the coboundary is bounded implies that this representation $\tilde{\rho}$ is the lifting of a representation into $\SO_0(2,2)$ whose bounded Euler class vanishes,
\ie that the group $\tilde{\rho}(\Gamma)$ preserves a closed achronal circle in $\wt\Ein_2$.

We claim that the existence of such an invariant achronal circle is equivalent to
the existence of a semi-conjugacy between $\rho_1$ and $\rho_2$ as stated in the conclusion
of Proposition \ref{pro:euPSL}. 

For the proof of this claim, it is convenient to consider the \textit{projectivized} anti-de Sitter and Einstein spaces,
\ie the quotients of $\AdS_3$ and $\Ein_2$ by $-\op{Id}$. The projectivized anti-de Sitter space is then
naturally identified with $\op{PSL}(2,\RR).$ According to
the identification between $(\op{Mat}(2, \RR), -\op{det})$ and $(\RR^{2,2}, \mathrm{q}_{2,2})$,
we obtain an identification between the projectivized Klein model $\overline{\Ein}_2$ and the
space of non-zero non-invertible 2-by-2 matrices up to a non-zero factor.
Such a class is characterized by the kernel and the image of its elements, \ie two lines in $\RR^2$.
In other words, $\overline{\Ein}_2$ is naturally isomorphic to the product $\RR\PP^1 \times \RR\PP^1$.
The conformal action of $\op{PO}(2,2) \approx \op{PSL}(2,\RR) \times \op{PSL}(2,\RR)$ on
$\RR\PP^1 \times \RR\PP^1$ is the obvious one:
$$(g_1, g_2).(x,y) = (g_1x, g_2y)$$
since the image of $g_1Ag_2^{-1}$ is the image by $g_1$ of the image of $A$, and
its kernel is the image under $g_2$ of the kernel of $A$.
The isotropic circles in $\overline{\Ein}_2 \approx \RR\PP^1 \times \RR\PP^1$ 
are the circles $\{ * \} \times \RR\PP^1$
and $\RR\PP^1 \times \{ * \}$. It follows quite easily that \textbf{acausal} circles in $\overline{\Ein}_2$
are graphs in $\RR\PP^1 \times \RR\PP^1$ of homeomorphisms from $\RR\PP^1 \to \RR\PP^1$.
\textbf{Achronal} circles are allowed to follow during some time one segment in 
$\{ * \} \times \RR\PP^1$ or $\RR\PP^1 \times \{ * \}$. It follows that they are fillings 
(cf. Remark~\ref{rk:filling}) of graphs of maps $f: \RR\PP^1 \to \RR\PP^1$ which are \textit{monotone,} \ie of degree $1$, preserving the cyclic order on $\RR\PP^1$, but which can be constant on some intervals and which can be non-continuous at certain points. In other words, $f$ lifts to
a non-decreasing map $\tilde{f}: \wt{\RR\PP}^1 \to \wt{\RR\PP}^1$. For more details on
this well-known geometric feature, we refer to \cite{mess1} or \cite{BBZ}.

In summary, we have proved that the representation 
$(\rho_1, \rho_2): \Gamma \to \op{PSL}(2,\RR) \times \op{PSL}(2,\RR) \approx \op{PO}(2,2)$
preserves a closed achronal circle $\Lambda$ in $\overline{\Ein}_2 \approx \RR\PP^1 \times \RR\PP^1$,
which is the filling of the graph of a monotone map $f: \RR\PP^1 \to \RR\PP^1$.
The invariance of $\Lambda$ means precisely that $f$ is $\Gamma$-equivariant: Proposition \ref{pro:euPSL} is proved.

\bibliography{quasfuchs}
\bibliographystyle{alpha}

\end{document}